%% file: main.tex
\documentclass[onefignum,onetabnum]{siamonline190516}

\input{ex_shared}

\ifpdf
\hypersetup{
  pdftitle={Analysis of a Computational Framework for Bayesian Inverse Problems},
  pdfauthor={ D. Sanz-Alonso and N. Waniorek}
}
\fi

\begin{document}
\def\argmindum{\mathop{\mbox{argmin}}}

\def\argmin#1{\argmindum_{#1}}

\maketitle

\begin{abstract}
This paper analyzes a popular computational framework to solve infinite-dimensional Bayesian inverse problems, discretizing the prior and the forward model in a finite-dimensional weighted inner product space. We demonstrate the benefit of working on a weighted space by establishing operator-norm bounds for finite element and graph-based discretizations of Mat\'ern-type priors and deconvolution forward models. For linear-Gaussian inverse problems, we develop a general theory to characterize the error in the approximation to the posterior. We also embed the computational framework into ensemble Kalman methods and MAP estimators for nonlinear inverse problems. Our operator-norm bounds for prior discretizations guarantee the scalability and accuracy of these algorithms under mesh refinement.
\end{abstract}



\section{Introduction}\label{sec:introduction} 
Bayesian inverse problems on infinite-dimensional Hilbert spaces arise in numerous applications, including medical imaging, seismology, climate science, reservoir modeling, and mechanical engineering \cite{calvetti2008hypermodels,bui2013computational,iglesias2015iterative,bigoni2020data,cleary2021calibrate}. In these and other applications, it is important to reconstruct function input parameters of partial differential equations (PDEs) based on noisy measurement of the PDE solution. This paper analyzes a framework to numerically solve infinite-dimensional Bayesian inverse problems, where the discretization is carried out in a weighted inner product space. The framework, along with a compelling demonstration of its computational benefits, was introduced in \cite{bui2013computational} to solve PDE-constrained Bayesian inverse problems using finite element discretizations. We develop a rigorous analysis that explains the advantage of working on a weighted space. Our theory accommodates not only finite element discretizations, but also graph-based methods from machine learning. For linear-Gaussian inverse problems, we bound the error in the approximation to the posterior. More broadly, we embed the discretization framework into ensemble Kalman methods and \emph{maximum a posteriori} (MAP) estimators for nonlinear inverse problems, and study the accuracy and scalability of these algorithms under mesh refinement.

An overarching theme of this paper is that,
by suitably choosing the weighted inner product, one can ensure accurate approximation of the infinite-dimensional Hilbert inner product as the discretization mesh is refined. Following this unifying principle, we analyze the finite element discretizations considered in \cite{bui2013computational} together with graph-based discretizations. Finite elements are predominant in scientific computing solution of PDE-constrained inverse problems, see e.g. \cite{kaipio2006statistical,bui2013computational,bigoni2020data}, whereas graph-based methods have been used, for instance, to solve inverse problems on manifolds \cite{harlim2020kernel,harlim2022graph,jiang2023ghost} and in machine learning applications in semi-supervised learning \cite{belkin2004semi,garcia2018continuum}. For both types of discretization, we obtain  operator-norm bounds for important classes of prior and forward models. Specifically, we will use Mat\'ern-type Gaussian priors and deconvolution forward models as guiding examples.

 Mat\'ern priors play a central role in Bayesian inverse problems \cite{roininen2014whittle}, spatial statistics \cite{stein2012interpolation}, and machine learning \cite{williams2006gaussian}. Efficient algorithms to sample Mat\'ern priors within the computational framework of \cite{bui2013computational} have recently been investigated in \cite{antil2022efficient}. Here, we obtain new operator-norm bounds for both finite element and graph-based discretizations. The former rely on classical theory, while the latter extend recent results from \cite{sanz2022spde}.
Deconvolution forward models arise in image deblurring and heat inversion, among other important applications. Since the seminal  paper by Franklin that introduced the formulation of Bayesian inversion in function space \cite{franklin1970well}, heat inversion has been widely adopted as a tractable testbed for theoretical and methodological developments, see e.g. \cite{stuart2010inverse,garcia2018continuum,trillos2017consistency,harlim2022graph}. 
Here, we derive operator-norm bounds for both finite element and graph-based discretizations. 

For linear-Gaussian inverse problems, operator-norm error bounds for prior and forward model discretizations translate into error bounds in the approximation of the posterior. We formalize this claim under a general assumption, which we verify for our guiding examples of Mat\'ern priors and deconvolution. Additionally, we apply the computational framework in \cite{bui2013computational} to algorithms for nonlinear inverse problems beyond the Markov chain Monte Carlo method for posterior sampling considered in \cite{petra2014computational}. Specifically, we investigate (i) ensemble Kalman methods \cite{iglesias2013ensemble,chada2020iterative} where we show, building on \cite{ghattas2022non}, that the effective dimension which determines the required sample size remains bounded along mesh refinements; and (ii) MAP estimation \cite{kaipio2006statistical,dashti2013map} where we show, building on \cite{ayanbayev2021convergence}, the convergence of MAP estimators under mesh refinement to the MAP estimator of the infinite-dimensional inverse problem. These results complement the vast literature on function-space sampling algorithms, see e.g. \cite{cotter2013,agapiou2017importance,ottobre2015function,trillos2017consistency}, and demonstrate that ensemble Kalman methods and MAP estimation can be scalable and accurate under mesh refinement.  We exemplify the new theory for MAP estimation in the nonlinear inverse problem of recovering the initial condition of the Navier-Stokes equations from pointwise observations of the velocity field \cite{cotter2009bayesian,cotter2010approximation,nickl2023posterior}. 

The well-posedness of the posterior measure under perturbations is one of the hallmarks of the Bayesian formulation of inverse problems \cite{stuart2010inverse,latz2020well,latz2023bayesian,sanzstuarttaeb}. For a fixed prior measure, the error in the posterior measure caused by discretization of the forward model can be bounded in Hellinger distance \cite{cotter2010approximation} and in Kullback-Leibler divergence \cite{marzouk2009stochastic}.
Posterior stability under perturbations to the prior measure in addition to perturbations to the likelihood have been recently investigated using Wasserstein distance \cite{sprungk2020local} and more general integral probability metrics \cite{garbuno2023bayesian}. These results hold even when the prior and the perturbation are mutually singular, as is the case for a discretization of the prior measure. In the linear-Gaussian setting, the Wasserstein distance bounds in \cite{sprungk2020local} yield a stability theory similar to ours; however, bounding the Wasserstein distance between the prior measures requires trace bounds on the covariance operators, whereas our results only necessitate operator norm bounds.
\color{black}
\subsection{Outline and Main Contributions}

\begin{itemize}[leftmargin=.09in]
    \item Section \ref{sec: formulation} reviews the formulation of a Bayesian inverse problem in a Hilbert space and presents our general discretization framework. We introduce the heat inversion problem with Mat\'ern Gaussian process prior as a model guiding example, and then illustrate how finite element and graph-based methods can be interpreted within our discretization framework.
    \item Section \ref{sec:unifiederroranalysis} presents a novel analysis of the error incurred in our discretization framework for linear-Gaussian Bayesian inverse problems. Theorems \ref{thm:posterior mean Theorem} and \ref{thm:posterior covariance Theorem} quantify the errors in the discretized posterior mean and covariance operators to their continuum counterparts, up to universal constants. We then verify the assumptions of this general theory for the finite element and graph-based discretizations of the heat inversion problem, thus deriving error bounds for these popular discretization schemes.
    \item Section \ref{sec:ensembleKalmanupdates} formulates the ensemble Kalman update within our general discretization framework. The results in this section give non-asymptotic bounds on the ensemble estimation of the posterior mean and covariance in terms of a notion of effective dimension based on spectral decay of the prior covariance operator. We then show that the effective dimension of the discretized prior covariance can be controlled by the effective dimension of the continuum covariance, which is necessarily finite. Consequently, the ensemble approximation will not deteriorate under mesh refinement.
    \item Section \ref{sec:MAPconvergence}
    considers posterior measures resulting from nonlinear Bayesian inverse problems and their MAP estimators. We show that with a suitable discretization of the forward model, the MAP estimators of the computationally tractable discretized posterior measures converge to the MAP estimators of the continuum posterior.  Finally, we apply the theory to the inverse problem of recovering the initial condition of the Navier-Stokes equations from pointwise observations. \color{black}
    \item Section \ref{sec:conclusions} closes with conclusions and directions for future work. 
\end{itemize}

\subsection{Notation} 
$\mathcal{S}_+^d$ denotes the set of $d\times d$ symmetric positive-semidefinite matrices and $\mathcal{S}_{++}^d$ denotes the set of $d\times d$ symmetric positive-definite matrices. Similarly, $\mathcal{S}_+^\H$ denotes the set of symmetric positive-semidefinite trace-class operators from a Hilbert space $\H$ to itself, and $\mathcal{S}_{++}^{\H}$ denotes the set of symmetric positive-definite trace-class operators from $\H$ to itself. Given two normed spaces $(X,\|\cdot\|_X)$ and $(Y,\|\cdot \|_Y)$, and a linear mapping $A:X\to Y$, we denote the operator norm of $A$ as $\|A\|_{op}:=\sup_{\|x\|_X=1}\|Ax\|_Y.$  $B(X,Y)$ denotes the space of all bounded linear operators from $X$ to $Y$.  Given two positive sequences $\{a_n\}$ and $\{b_n\}$, the relation $a_n\lesssim b_n$ denotes that $a_n \leq c b_n$ for some constant $c>0$. $\mathbb{1}_B$ denotes the indicator of the set $B$. Given a matrix $A\in \mathcal{S}_{++}^d,$ we denote the weighted inner product as $\langle \vec{u},\vec{v}\rangle_A:=\vec{u}^TA\vec{v}$, and the corresponding weighted norm as $\|\vec{u}\|_A:=\sqrt{\vec{u}^TA\vec{u}}$. Finally, $\|\cdot \|_2$ denotes the usual Euclidean norm on $\R^d.$
\section{Problem Setting and Computational Framework}\label{sec: formulation}
This section contains necessary background. Subsection \ref{ssec:Hilberspacesetting} overviews the formulation of Bayesian inverse problems in Hilbert space. Subsection \ref{ssec:computationalframework}  describes the computational framework analyzed in this paper. Finally, Subsection \ref{ssec:functionspacesetting} shows how finite element and graph-based methods for inverse problems in function space can be viewed as particular instances of the general framework. 

\subsection{Inverse Problem in Hilbert Space}\label{ssec:Hilberspacesetting}
Let $\H$ be an infinite-dimensional separable Hilbert space with inner product $\langle \cdot, \cdot \rangle_\H.$ Consider the inverse problem of recovering an unknown $u\in \H$ from data $y \in \R^{d_y}$ related by
\begin{align}\label{eq:IPinHilbertSpace}
    y=\mathcal{F}u+\eta, \quad \eta\sim \Nc(0,\Gamma),
\end{align}
where $\mathcal{F}: \H \to \R^{d_y}$ is a linear and bounded forward model, and $\eta$ represents Gaussian observation noise with known covariance matrix $\Gamma\in S_{++}^{d_y}$. Nonlinear forward models will be considered in Section \ref{sec:MAPconvergence}. The observation model \eqref{eq:IPinHilbertSpace} implies a Gaussian likelihood function
\begin{align}\label{eq:infinitelikelihood}
    \pi_{\text{like}}(y|u)\propto \exp \left(-\frac{1}{2}\|y-\mathcal{F}u\|_{\Gamma^{-1}}^2 \right).
\end{align}
We adopt a Bayesian approach \cite{tarantola2005inverse,kaipio2006statistical,sanzstuarttaeb} and let $\mu_0 = \Nc(m_0, \Cm_0)$ be a Gaussian prior measure on $\H,$ where $\Cm_0: \H \to \H$ is a trace-class covariance operator defined by the requirement that
\begin{align}\label{eq:covInHilbert}
    \langle v,\mathcal{C}_0w\rangle_\H=\mathbb{E}^{u \sim \mu_0}\Bigl[ \bigl\langle v,(u-m_0) \bigr\rangle_\H \bigl\langle (u-m_0),w \bigr\rangle_\H \Bigr], \quad \quad \forall v, w \in \H,
\end{align}
and $m_0 \in \H$ is assumed to belong to the Cameron-Martin space $E = \text{Im}(\Cm_0^{1/2}) \subset \H.$
The Bayesian solution to the inverse problem is the conditional law of $u$ given $y,$ called the posterior probability measure $\mu_{\text{post}}.$
Application of Bayes' rule in infinite dimension \cite{stuart2010inverse} characterizes the posterior as a change of measure with respect to the prior, with Radon-Nikodym derivative given by the likelihood:
\begin{align}\label{eq:changemeasureinfinite}
    \frac{d\mu_{\text{post}}}{d\mu_0}(u)=\frac{1}{Z}\pi_{\text{like}}(y|u),
\end{align}
where $Z=\int \pi_{\text{like}}(y|u) \, d\mu_0$ is a normalizing constant. Since the prior is Gaussian and the forward model is linear and bounded, the posterior is also Gaussian, $\mu_{\text{post}}=\Nc(m_{\text{post}},\mathcal{C}_{\text{post}}),$ where
\begin{align}\label{eq:infiniteMean}
     m_{\text{post}} &=m_0+\mathcal{C}_0\F^*(\F\mathcal{C}_0\F^*+\Gamma)^{-1}(y-\F m_0),\\
    \label{eq:infiniteCov}
     \Cm_{\text{post}} &=\Cm_0-\Cm_0\F^*(\F\Cm_0\F^*+\Gamma)^{-1}\F\Cm_0.
\end{align}
Here, $\F^*$ denotes the adjoint of $\F,$ which is the unique map from $\R^{d_y}$ to $ \H$ that satisfies
\begin{align*}
    \langle \F u, y\rangle =\langle u,\F^*y\rangle_\H, \quad \quad \forall u \in \H, y \in \R^{d_y}.
\end{align*}
In Subsection \ref{ssec:computationalframework}, we will review the computational framework in \cite{bui2013computational} to approximate the posterior. The idea is to replace the inverse problem \eqref{eq:IPinHilbertSpace} on the infinite-dimensional Hilbert space $\H$ with an inverse problem on a finite-dimensional weighted inner product space. Under general conditions on the discretization of the forward model and prior measure, the posterior mean and covariance of the finite and infinite-dimensional inverse problems are close together; this claim will be formalized and rigorously established in Section \ref{sec:unifiederroranalysis}.

\subsection{Computational Framework}\label{ssec:computationalframework} This subsection overviews the computational framework presented in \cite{bui2013computational}. 
Let $\{\oldphi_1,\oldphi_2,\ldots,\oldphi_n\}$ be a basis for an $n$-dimensional space $\V\subset \H.$ 
For $u=\sum_{i=1}^n u_i\oldphi_i \in \V$, we denote by $\vec{u}=(u_1,\ldots,u_n)^T$ the vector of coefficients of $u$ in this basis.  We will endow the space of coefficients with a weighted inner product that is naturally inherited from the inner product in $\H.$ Specifically, for any $u,v\in \V$, we have
$$
\langle u,v\rangle_{\H}=\vec{u}^TM\vec{v}=\langle \vec{u},\vec{v}\rangle_M,
$$
where the matrix $M=(M_{ij})_{i,j =1}^n$ is given by
$$\label{eq:mass}
M_{ij}=\langle \oldphi_i, \oldphi_j \rangle_\H, \quad i,j\in \{1,\ldots,n\}.
$$
This observation motivates us to introduce the inner product space $\R^n_M$ defined by vector space $\R^n$ and inner product $\langle \cdot, \cdot \rangle_M.$ 
We remark that for both the finite elements and graph-based discretizations introduced in Subsections \ref{sec:FEM discretization} and \ref{sec:graph discretization}, the corresponding basis $\{\oldphi_i\}$ will not be orthonormal; consequently  $M\neq I$ and $\R^n_M$ will not agree with the standard Euclidean inner product space.  \nc 
We now aim to replace the inverse problem \eqref{eq:IPinHilbertSpace} with an inverse problem on  the weighted space $\R^n_M$.

To begin with, we define a discretization map $P:\H \to \R^n_M$ that assigns to $u\in \H$ the vector $\vec{u}^\star \in \R^n_M$ satisfying
\begin{equation}\label{eq:discretizationmap}
Pu=\vec{u}^\star=(u_1^\star,\ldots,u_n^\star)^T=\arg \min_{u_1,\ldots,u_n}\frac{1}{2} \Bigl\|u - \sum_{i=1}^nu_i\oldphi_i \Bigr\|_{\H}.
\end{equation}
In practice, this discretization can be computed by solving the linear system $M\vec{u}^\star=\vec{b}$, where $b_i=\langle u,\oldphi_i\rangle_{\H}$. Conversely, our theoretical analysis will rely on an extension map $P^*:\R^n_M\to \V \subset \H$ that assigns to a vector $\vec{u}\in \R^n_M$ the element $P^*u \in \V$ defined by
\begin{equation}\label{eq:extensionmap}
P^* \vec{u}=\sum_{i=1}^n u_i \oldphi_i.
\end{equation}
One can verify that $\langle P u, \vec{v}\rangle_{M}=\langle u,P^*\vec{v}\rangle_{\H},$ so that $P^*$ is the adjoint of $P,$ as our notation suggests. Similarly, one can verify that the map $P^*P:\H\to \V\subset \H$ is the orthogonal projection onto  $\V,$ so that $PP^*:\R_M^n\to \R_M^n$ is the identity map.

We are ready to introduce the inverse problem on the weighted inner product space $\R^n_M.$ In analogy with \eqref{eq:IPinHilbertSpace}, we seek to recover $\vec{u} \in \R^n_M$ from data $y$ related by
\begin{equation}
    y = F \vec{u} + \eta, \quad \eta\sim \Nc(0,\Gamma),
\end{equation}
where $F: \R^n_M \to \R^{d_y}$ is a discretized forward model, identified with a matrix $F\in \R^{d_y \times n}$. (Here and henceforth we will abuse notation and identify linear maps and matrices without further notice.) The map $F$ should approximate $\F$ in the sense that $\| \F - FP \|_{op}$ is small; we refer to Subsections \ref{ssec:errorforwardFEM} and \ref{ssec:errorforwardgraph} for examples arising from discretization of PDE-constrained forward models in function space. In analogy with \eqref{eq:infinitelikelihood}, the likelihood of observed data $y$ given a coefficient vector $\vec{u}\in \R_M^n$ is given by
$$
\pi_{\text{like}} (y|\vec{u})\propto \exp \left(-\frac{1}{2}\|y-F\vec{u}\|_{\Gamma^{-1}}^2 \right).
$$
Following the analogy with Subsection \ref{ssec:Hilberspacesetting}, we choose a Gaussian prior $\mu_0^n = \Nc( \vec{m}_0, C_0)$ on the space $\R^n_M,$
with Lebesgue density
\begin{align}\label{eq:finitePrior}
    \pi_0(\vec{u})\propto \exp \left( -\frac{1}{2} \bigl\langle \vec{u}-\vec{m}_0, C_0^{-1} (\vec{u} - \vec{m}_0) \bigr\rangle_M \right).
\end{align}
For the prior mean, we take $\vec{m}_0 = Pm_0;$ for the prior covariance, we may take an invertible operator $C_0: \R^n_M \to \R^n_M$ that approximates well $\Cm_0,$ in the sense that $\| \Cm_0 - P^* C_0 P \|_{op}$ is small; we refer to Subsections \ref{ssec:errorcovarianceFEM} and \ref{ssec:errorcovariancegraph} for examples arising from discretization of Mat\'ern-type Gaussian processes.  
Similar to \eqref{eq:covInHilbert}, the prior covariance $C_0$ satisfies, by definition, that
$$ \label{eq:gaussianWeighted}
\langle \vec{v},C_0\vec{w}\rangle_M=\mathbb{E}^{\vec{u} \sim \mu_0^n}\Bigl[\langle \vec{v},(\vec{u}-\vec{m})\rangle_M \bigl\langle (\vec{u}-\vec{m}),\vec{w} \bigr\rangle_M \Bigr], \quad \quad \forall \vec{v}, \vec{w} \in \R^n_M.
$$
While we view $\mu_0^n$ as a Gaussian measure on $\R^n_M,$ it is helpful to notice that as a multivariate Gaussian distribution on $\R^n$ equipped with the standard inner product, we have that $\mu_0^n = \Nc( \vec{m_0}, C_0^E),$ with covariance $C_0^E = C_0 M^{-1}.$ 

Applying Bayes' rule as in \eqref{eq:changemeasureinfinite}, we can express the posterior as a change of measure with respect to the prior:
\begin{align}\label{eq:changemeasurefinite}
    \frac{\mu_{\text{post}}^n}{\mu_0^n}=\frac{1}{Z_h}\pi_{\text{like}}(y|\vec{u}).
\end{align}
Since the space $\R^n_M$ is finite dimensional, we recover the standard Bayes' formula for the posterior distribution with Lebesgue density given by
\begin{align}\label{eq:finitePost}
\pi^n_{\text{post}}(\vec{u}|y)\propto \exp \left(-\frac{1}{2}\|y-F\vec{u}\|_{\Gamma^{-1}}^2 -\frac{1}{2} \langle \vec{u}-\vec{m}_0, C_0^{-1} (\vec{u} - \vec{m}_0) \rangle_M \right).
\end{align}
As in the infinite-dimensional setting, the posterior $\mu_{\text{post}}^n = \Nc(\vec{m}_{\text{post}}, C_{\text{post}})$ is Gaussian with analogous expressions for its mean $\vec{m}_{\text{post}}$ and covariance $C_{\text{post}}$. A note of caution is that in the weighted space $\R^n_M$ a careful distinction must be made between adjoint and matrix transpose. We present a detailed discussion of this distinction in Appendix \ref{sec:adjoint discussion}. In particular, 
if the forward model is given by matrix $F\in \R^{d_y \times n},$ then the adjoint map is given by $F^{\natural}=M^{-1}F^T \in \R^{n \times d_y}.$  Thus,
\begin{align}\label{eq:finiteMean}
    \vec{m}_{\text{post}} &=\vec{m}_0+C_0F^{\natural}(FC_0F^{\natural}+\Gamma)^{-1}(y-F\vec{m}_0),\\
    \label{eq:finiteCov}
    C_\text{post}&=C_0-C_0F^{\natural}(FC_0F^{\natural}+\Gamma)^{-1}FC_0.
\end{align}
One can also view the posterior as a Gaussian measure in the standard Euclidean space, in which case the covariance operator is given by $C_{\text{post}}^E=C_{\text{post}}M^{-1}$. For completeness, we include the expressions for the posterior mean and covariance in Euclidean space written in terms of the Euclidean prior covariance:
\begin{align}\label{eq:euclideanmean}
    \vec{m}_{\text{post}}&=\vec{m}_0+C_0^EF^T(FC_0^EF^T+\Gamma)^{-1}(y-F\vec{m}_0),\\
\label{eq:euclideancovariance}
    C_{\text{post}}^E &=C_0^E-C_0^EF^T(FC_0^EF^T+\Gamma)^{-1}FC_0^E. 
\end{align}
Notice that both pairs of equations \eqref{eq:finiteMean}--\eqref{eq:finiteCov} and \eqref{eq:euclideanmean}--\eqref{eq:euclideancovariance} are analogous to their infinite-dimensional counterparts in \eqref{eq:infiniteMean}--\eqref{eq:infiniteCov}. However, our theory and examples in the next subsection will demonstrate the advantage of working on the weighted inner product space. Specifically, we will show that finite element and graph discretizations of important classes of prior covariance and forward model give discretized quantities $F^\natural$ and $C_0$ that approximate well in operator norm their infinite-dimensional counterparts $\F^*$ and $\Cm_0$; the same would not be true for the Euclidean analogs $F^T$ and $C_0^E.$

\subsection{Inverse Problems in Function Space}\label{ssec:functionspacesetting}
As particular instances of the computational framework overviewed in the previous subsection, here we consider discretization of inverse problems in function space using finite elements and graphs. To illustrate the ideas, we focus on inverse problems with Matérn-type priors and deconvolution forward models, introduced in Examples \ref{ex:prior} and \ref{ex:forward} below. We let $\H = L^2(\Omega)$ be the Hilbert space of square integrable functions on $\Omega$ with the usual inner product $\langle \cdot, \cdot \rangle_{L^2(\Omega)}.$ For finite element discretizations, we take $\Omega = \D$ to be a sufficiently regular domain $\D \subset \R^d,$ while for graph discretizations we take $\Omega = \M$ to be a $d$-dimensional smooth, connected, compact Riemannian manifold embedded in $\R^{D}.$ These choices of domain $\Omega$ reflect popular settings for finite element and graph-based discretizations in Bayesian inverse problems.

\begin{table}[H]
	\begin{center}
		\begin{tabular}{ | c | c |c | c |}
			\hline
			Domain & Basis  & Discretization   & Error bounds  \\ \hline
		 $\D$& Lagrange finite elements &  Subsection \ref{sec:FEM discretization}  & Subsection \ref{ssec:errorFEM} \\ \hline
			 $ \M$ & Random geometric graphs  & Subsection \ref{sec:graph discretization}    & Subsection \ref{ssec:errorgraph}   \\ \hline
		\end{tabular}
		\caption{Roadmap of the discretizations analyzed in this paper as particular examples of the general computational framework introduced in Subsection \ref{ssec:computationalframework}.}		
	\end{center}
\end{table}
We next introduce our examples of  prior and forward models, followed by a discussion of their finite element and graph discretizations, both of which will be analyzed in a unified way in Section \ref{sec:unifiederroranalysis}. 

\begin{example}[Matérn-Type Prior]\label{ex:prior}
We will consider Gaussian priors  $\mu_0=\Nc(m_0,\mathcal{C}_0)$  with Mat\'ern covariance operator $\Cm_0$ and mean
    $m_0$ in the Cameron-Martin space $\emph{Im}(\Cm_0^{1/2})$. Specifically, for positive integer $\alpha,$ we define $\Cm_0 = \mathcal{A}^{-\alpha},$ where  $\mathcal{A}^{-1}:L^2(\Omega)\to H^2(\Omega)$ is the solution operator for the following elliptic PDE in weak form:
\begin{align}\label{eq:ellipticPDE}
 \int_{\Omega}\Theta \nabla u \cdot \nabla p + \int_{\Omega} bup =\int_{\Omega} f p\, ,  \quad \quad \forall p\in H^1(\Omega).
\end{align}
 That is, for $f\in L^2(\Omega)$, $\mathcal{A}^{-1}f=u$, where $u$ solves \eqref{eq:ellipticPDE}. Here and below, integrals are with respect to the Lebesgue measure if $\Omega = \D$ and with respect to the Riemannian volume form if $\Omega = \M.$
In the case $\Omega = \D,$ we assume for concreteness Dirichlet boundary conditions and take $\H^1(\Omega) \equiv H_0^1(\D).$ In the case $\Omega = \M,$ we take $\H^1(\Omega) \equiv H^1(\M).$
 We emphasize that we view $\mathcal{A}^{-1}$ as a map from the right-hand side of the elliptic PDE to the solution. To ensure that a solution to the PDE exists, is unique, and is sufficiently regular, the coefficients $\Theta$ and $b$ must be positive and sufficiently smooth; precise assumptions will be given in our theorem statements. Choosing these coefficients to be spatially varying allows one to encode additional prior information about the unknown, such as anisotropic correlations. The parameter $\alpha$ in the covariance $\Cm_0 = \mathcal{A}^{-\alpha}$ controls the regularity of prior draws. For simplicity, we restrict our attention to positive integer-valued $\alpha$, but extensions to fractional values are possible \cite{bolin2020rational,antil2022efficient}.
 If the domain $\Omega$ is sufficiently regular, $\alpha$ is chosen to be large enough, and $\Theta,b$ are positive and sufficiently smooth, then the covariance operator $\mathcal{C}_0$ is trace-class on $L^2(\Omega)$ \cite{stuart2010inverse}.
\end{example}

\begin{example}[Deconvolution Forward Model]\label{ex:forward} 
Consider the heat equation:
\begin{align}\label{eq:heatEqn}
    \begin{cases}
    v_t(x,t)=\Delta v(x,t), &  x\in \Omega, \, t>0,\\
    v(x,0)=u(x), & x\in \Omega.
    \end{cases}
\end{align}
In the case $\Omega = \D,$ we supplement \eqref{eq:heatEqn} with Dirichlet boundary conditions $ v  (x,t)=0, \, x\in \partial \D.$
We let $\G:L^2(\Omega)\to L^2(\Omega)$ map an input function $u\in L^2(\Omega)$ to the solution to \eqref{eq:heatEqn} at time $t=1.$  
It is well known that the heat equation admits a variational characterization, which will lend itself more naturally to the numerical approximations introduced in Subsections \ref{sec:FEM discretization} and \ref{sec:graph discretization}.
We consider a linear observation model given by local averages of the solution over a Euclidean ball of radius $\delta>0$, which approximate pointwise observations:
\begin{align}\label{eq:local averages}
    \mathcal{O}v=\left[v^{\delta}(x_1),\ldots,v^{\delta}(x_{d_y}) \right]^T, \quad v^{\delta}(x_i):=\int_{B_{\delta}(x_i) \cap \Omega}v(x,1).
\end{align}
We then write our forward model as
\begin{align}\label{eq:forward model}
    \F:=\mathcal{O}\circ\G.
\end{align}
The utility of considering the local average observation map is that it is a bounded linear map from $L^2(\Omega)$ to $\R^{d_y}$, whereas pointwise observations are only bounded from subspaces of $L^2(\Omega)$ with enough regularity to guarantee almost everywhere continuity. While the forward map given by the heat equation is sufficiently smoothing to guarantee the solution is continuous, the graph-based approximations to the forward map that we will consider may only converge in $L^2(\Omega)$. In those cases, it will be more convenient to consider the local average observations, as in \cite{garcia2018continuum}. 
\end{example}

\subsubsection{Finite Element Discretizations} \label{sec:FEM discretization}
The use of finite elements within the computational framework overviewed in Subsection \ref{ssec:computationalframework} was proposed and numerically investigated in  \cite{bui2013computational}. Here, the domain $\Omega$ is an open, bounded, and sufficiently regular subset $\D\subset \R^d.$ We let  $\V$  be a finite element discretization space with basis functions denoted by $\{\oldphi_j\}_{j=1}^n$. For concreteness, we restrict our attention to linear Lagrange polynomial basis vectors. These basis functions correspond to nodal points $\{x_j\}_{j=1}^n\in \R^d$ such that $\oldphi_i(x_j)=\delta_{ij}$ for $i,j\in \{1,\ldots,n\}.$ The domain $\D$ is partitioned by a mesh into elements $e_1,e_2,\ldots$ with $h_i:=\text{diameter}(e_i)$ and $h:=\max h_i.$ Two canonical families of Lagrange elements include $d$-simplexes and $d$-hypercubes. A $d$-simplex is the region $e\subset\R^d$ determined by $d+1$ distinct points (for $d=1$ this is an interval and for $d=2$ a triangle). A $d$-hypercube is a product of $d$ intervals on the real line, $e = [a_1,b_1]\times \cdots \times [a_d,b_d]\subset \R^d$. As the mesh is refined (that is, as we take $h\to 0$), the dimension of $\V$ grows. To emphasize that we are considering a family of subspaces indexed by the mesh width parameter $h$, we adopt the notation $\V=\V_h$, as is typical in the finite element literature. \color{black} The theory we will present in Subsection \ref{ssec:errorFEM} holds for uniform (and more generally for quasi-uniform) mesh refinements. 
\begin{example}[Finite Element Approximation of Matérn Covariance]
For an operator $\mathcal{B}:L^2(\D)\to L^2(\D)$, one can consider the action of this operator restricted to the subspace  $\V_h$, given by $\mathcal{B}_h=P^*P\mathcal{B}P^*P:\V_h\to \V_h$, which is a finite-dimensional linear map. The matrix representation $B$ of this operator must satisfy
$$
\int_{\D}\oldphi_i \mathcal{B} \oldphi_j \, dx=\langle \vec{e}_i,B\vec{e}_j\rangle_M,
$$
where $\vec{e}_i$ is the coordinate vector that corresponds to $\oldphi_i.$ This implies that $B:\R_M^n\to \R_M^n$ can be written explicitly as
\begin{align}
    B=M^{-1}K,
\end{align}
where 
\begin{align}
    K_{ij}=\int_{\D}\oldphi_i \mathcal{B} \oldphi_j \, dx, \quad i,j\in \{1,\ldots,n\}.
\end{align}
In particular, when representing the differential operator $\mathcal{A}$ defined in \eqref{eq:ellipticPDE} we get that $K$ is given by the stiffness matrix with entries
\begin{align}\label{eq:stiffness}
    K_{ij}=\int_{\D}\Theta(x) \nabla \oldphi_i(x) \cdot \nabla \oldphi_j(x)+b(x)\oldphi_i(x)\oldphi_j(x) \, dx, \quad i,j\in \{1,\ldots,n\}.
\end{align}
One can verify that both $A=M^{-1}K$ and $A^{-1}=K^{-1}M$ are self-adjoint in the weighted inner product space. Recall that the coefficients of the Galerkin approximation to the elliptic PDE \eqref{eq:ellipticPDE} are given by solving the linear system 
\begin{align*}
    K\vec{u}=\vec{b},
\end{align*}
where $b_i=\int_{\D} f\oldphi_i \, dx$ for $1\leq i\leq n,$ see \cite{oden2012introduction}. Since we have that $\vec{b}=M\vec{f}$ (where $\vec{f}=Pf$), we see that $A^{-1}\vec{f}=K^{-1}M\vec{f}$ is precisely the Galerkin approximation to $\mathcal{A}^{-1}f$.  We will show in Subsection \ref{ssec:errorcovarianceFEM} that the matrix $C_0:\R^n_M\to \R^n_M$ given by $C_0=A^{-\alpha}=(K^{-1}M)^\alpha$ approximates well the continuum Mat\'ern covariance operator $\Cm_0.$    
\end{example}

\begin{example}[Finite Element Approximation of Heat Forward Model]
The variational formulation of the heat equation \eqref{eq:heatEqn} naturally leads to a \textit{semidiscrete} Galerkin approximation
\begin{align}
    v_h(x,t)=\sum_{i=1}^nA^i(t)\oldphi_i(x),
\end{align}
which is required to satisfy that, for all $\oldphi\in  \V_h,$
\begin{alignat}{2}\label{eq:semidiscrete galerkin}
&\int_{\D} \frac{\partial v_h}{\partial t} \oldphi(x) \, dx +\int_{\D} \nabla v_h \cdot \nabla \oldphi(x) \, dx=0, \quad  \quad  &&t\in (0,1],\\
&\int_{\D}v_h(x,0)\oldphi(x) \, dx =\int_{D}u(x)\oldphi(x) \, dx, \quad  &&t=0.
\end{alignat}
This leads to a system of differential equations for the time-dependent coefficients $A^i(t):$
\begin{alignat}{2} \label{eq:semidiscrete ODE}
    &\sum_{j=1}^nM_{ij}\frac{\partial}{\partial t}A^j(t)+K_{ij}A^j(t)=0, \quad \quad &i=1,\ldots,n, \\
    &\sum_{j=1}^n M_{ij}A^j(0)=g_i, & i=1,\ldots,n,
\end{alignat}
 where $M_{ij}=\langle \oldphi_i, \oldphi_j\rangle_{L^2(\D)}$, $K_{ij}=\langle \oldphi_i, -\Delta \oldphi_j\rangle_{L^2(\D)}$, and $g_i=\langle u,\oldphi_i\rangle_{L^2(\D)}.$ Note that $\vec{A}(t)=\left(A^1(t),\ldots,A^n(t)\right)^T\in \R^n_M.$ Given the eigenpairs of  $K$ and $M$,  this system of differential equations can  be solved analytically. However, computing the eigenpairs may be prohibitively expensive in practice, so the solution is often approximated by a time-discretization of the semidiscrete problem. One such classical time discretization of \eqref{eq:semidiscrete galerkin} is the Crank-Nicolson method: find $v_h^k\in V_h,$ for $k=0,\ldots,1/\Delta t$ such that, for every $\oldphi\in \V_h,$
\begin{alignat}{2} \label{eq:time discrete galerkin}
&\int_{\D}\frac{v_h^k(x)-v_h^{k-1}(x)}{\Delta t}\oldphi(x) \, dx+\int_{\D}\nabla \left( \frac{v_h^k(x)+v_h^{k-1}(x)}{2}\right)\cdot \nabla \oldphi(x) \, dx =0, \quad k=1,\ldots,1/\Delta t,\\
&\int_{\D} v_h^0(x) \oldphi(x) \, dx =\int_{\D}u(x)\oldphi(x)\, dx.
\end{alignat}
This leads to a system of linear equations analogous to \eqref{eq:semidiscrete ODE} for the coefficients $A^j_k\approx A^j(k\Delta t)$ for $k=0,\ldots, 1/\Delta t-1$:

\begin{alignat*}{2}
        \frac{M+K\Delta t}{2}\vec{A}_{k+1}&=\frac{M-K\Delta t}{2}\vec{A}_{k}, \quad \quad  &&k=1,\ldots,1/\Delta t-1,\\
        M\vec{A}_0&=\vec{g},  \quad &&k=0.
\end{alignat*}
The second equation is simply the requirement that $\vec{A}_0$ be given by the orthogonal projection onto $\V_h$ as $\vec{A}_0=Pu.$ We write the mapping from the coefficients of the initial condition to the coefficients of the solution at $t=1$ as $G:\R^n_M\to \R^n_M$, where
\begin{align}
    G=\left(\left(\frac{M+K\Delta t}{2}\right)^{-1}\left(\frac{M-K\Delta t}{2}\right) \right)^{\frac{1}{\Delta t}}.
\end{align} 
We denote by $O:\R_M^n\to \R^{d_y}$ the map from coefficients to observations. Our discretized forward model $F:\R^n_M\to \R^{d_y}$ is then given by
\begin{align}\label{eq:discrete FEM forward map}
F:=OG.
\end{align}
We remark that the observation map $\mathcal{O}$ applied to the numerical solution to the heat equation is an integral of a piecewise linear function. Such an integral can be computed exactly via quadrature methods, which are typically tractable in low dimensions (e.g. $d=1,2,3$). Since finite element methods are seldom implemented in higher dimensions in practice, we make the simplifying assumption that the observation map can be evaluated exactly for functions in $\V_h$. 
\end{example}

\subsubsection{Graph-Based Discretizations}\label{sec:graph discretization}
We will now see how the graph-based discretizations considered in \cite{garcia2018continuum} fit into the above framework.  Here we let $\Omega=\M$ be a $d$-dimensional smooth, connected, compact, Riemannian manifold embedded in $\R^{D}.$ We assume further that $\M$ has bounded sectional curvature and Riemannian metric inherited from $\R^{D}$. Without loss of generality, we let $\M$ be normalized such that $vol(\M)=1.$ We assume we have access to a point cloud $\M_n=\{x_1,\ldots,x_n\}\subset \M$ of $n$ samples from the uniform distribution $\gamma$ on $\M$. We let $\gamma_n=\frac{1}{n}\sum_{i=1}^n\delta_{x_i}$ denote the empirical measure of $\M_n$. We will show that this setting naturally fits into our computational framework with basis vectors of the form $\oldphi_i(x)=\mathbb{1}_{U_i}(x)$, where the sets $\{U_i\}_{i=1}^n\subset \M$ form a partition of $\M$. We recall the following result from \cite{garcia2020error}:
\begin{proposition}[Existence of Transport Maps]
\label{thm:transportMaps}
    There is a constant $c$ such that, with probability one, there exists a sequence of transport maps $T_n:\M\to \M_n$ such that $\gamma_n=T_{n\sharp}\gamma$ and
    \begin{align}
        \lim_{n\to \infty}\sup \frac{n^{1/d}\sup_{x\in \M}d_{\M}(x,T_n(x))}{(\log n)^{c_d}}\leq c,
    \end{align}
    where $c_d=3/4$ if $d=2$ and $c_d=1/d$ otherwise.
\end{proposition}
In the above, $\gamma_n=T_{n\sharp}\gamma$  indicates that $\gamma(T_n^{-1}(U))=\gamma_n(U)$ for all measurable $U$, and $d_{\M}$ denotes the geodesic distance. We denote the pre-image of each singleton as $U_i=T_n^{-1}(\{x_i\})$. The sets $U_i$ form a partition of $\M$ and by the measure preserving property of $T_n$ all have mass $\gamma(U_i)=1/n.$ It follows from Proposition \ref{thm:transportMaps} that we have $U_i\subset B_{\M}(x_i,\eps_n)$, where 
$$
\eps_n:= \sup_{x\in \M}d_{\M}(x,T_n(x))\lesssim \frac{(\log n)^{c_d}}{n^{1/d}}.
$$
These sets will be used to define our basis functions: $\oldphi_i(x)=\mathbb{1}_{U_i}(x)$, which correspond to locally constant interpolations of functions in $L^2(\M).$ We can then interpret vectors of function values on the point cloud $\M_n$ as coefficients vectors in $\R_M^n$, where the corresponding mass matrix is $M=\frac{1}{n}I_n.$

\begin{example}[Graph Approximation of Matérn Covariance]
In the manifold setting, we opt to approximate the elliptic differential operator via the graph Laplacian. We define a symmetric weight matrix $W\in \R^{n\times n}$ that has entries $W_{ij}\geq 0$ corresponding to closeness between points in $\M_n$ by 
\begin{align}
    \label{weights}
    W_{ij}=\frac{2(d+2)}{n\nu_d h_n^{d+2}}\mathbb{1}_{\{\|x_i-x_j\|_2<h_n\}}.
\end{align}
Here $\nu_d$ denotes the volume of the $d$-dimensional unit ball, and $h_n$ denotes the graph connectivity. Defining the degree matrix $D =\emph{diag}(d_1,\ldots,d_n)$ where $d_i=\sum_{j=1}^nW_{ij},$ the unnormalized graph Laplacian is given by $\Delta_n^{un}=D-W.$ This operator can be shown to converge pointwise and spectrally to $-\Delta$ (the Laplace-Beltrami operator) in the manifold setting \cite{belkin2005towards,burago2015graph}. This construction can also be generalized to elliptic operators with spatially varying coefficients. To do so, we define a nonstationary weight matrix $\tilde{W}$ with entries
\begin{align}
\tilde{W}_{ij}=W_{ij}\sqrt{\Theta(x_i)\Theta(x_j)}=\frac{2(d+2)}{n\nu_d h_n^{d+2}}\mathbb{1}_{\{\|x_i-x_j\|_2<h_n\}}\sqrt{\Theta(x_i)\Theta(x_j)} \,.
\end{align}
We then define $\Delta^{\Theta}_n=\tilde{D}-\tilde{W}$ where $\tilde{D}_{ii}=\sum_{j=1}^n\tilde{W}_{ij}$. Denoting $B_n=\emph{diag}\bigl(b(x_1),\ldots,b(x_n)\bigr)$, we approximate the differential operator $\mathcal{A}$ as the matrix $A:\R_M^n\to \R_M^n$ given by
\begin{align}\label{eq:graphApprox}
A=\Delta^{\Theta}_n+B_n.
\end{align}
Both $A$ and $A^{-1}$ are symmetric and self-adjoint in the weighted inner product space. The matrix $C_0:\R_M^n\to \R_M^n$ given by $C_0=A^{-\alpha}$ will be shown to approximate well the continuum Mat\'ern covariance operator $\Cm_0$ in Subsection \ref{ssec:errorcovariancegraph}.
\end{example}

\begin{example}[Graph Approximation of Heat Forward Model]
For our graph-based approximation of the heat forward model, we consider the differential equation
\begin{align}
    \begin{cases}\label{eq:graph heat equation}
        \frac{\partial v_n}{\partial t} = -\Delta_n^{un}v_n, & t>0,\\
        v_n(0)=u_n.
    \end{cases}
\end{align}
We let $G_n:\R_M^n\to \R_M^n$ map $u_n\in \R_M^n$ to the solution $v_n(1)$ to \eqref{eq:graph heat equation}. Given the eigenpairs of the graph Laplacian, denoted $\bigl\{(\lambda_n^{(i)},\psi_n^{(i)} \bigr\}_{i=1}^n$, we can explicitly write the solution to this ODE:
\begin{align}
    G_n(u_n)=\sum_{i=1}^n\exp\left(-\lambda_n^{(i)}\right)\langle u_n,\psi_n^{(i)}\rangle_M\psi_n^{(i)}.
\end{align}
The computations presented above for both the prior covariance discretization and the heat forward model do not require the explicit computation of the sets $U_i$ that partition $\M$ and define our basis vectors. In order to exactly compute the observation map defined in \eqref{eq:local averages}, we would need to exactly know the sets $U_i$. To circumvent this challenge, we instead  consider a surrogate observation map $\mathcal{O}_n:\R_M^n\to \R^{d_y}$ that only requires knowledge of the point cloud $\M_n:$
\begin{align}\label{eq:surrogate observations}
    \mathcal{O}_nv_n=\left[\bar{v}^{\delta}_{n,1},\ldots ,\bar{v}^{\delta}_{n,1} \right]^T, \quad \bar{v}^{\delta}_{n,j}:=\frac{1}{n}\sum_{k:x_k\in B_{\delta}(x_j)\cap \M_n}[v_n(1)]_k, 
\end{align}
for $1\leq j \leq d_y,$ where $[v_n(1)]_k$ denotes the $k^{th}$ entry of the vector $v_n$.  For the purposes of our analysis, it will be necessary to make the assumption on the boundary of $\B_\delta(x_k)\cap \M$. In particular, we assume that $\text{length}\left(\partial (\B_\delta(x_k)\cap \M) \right)=C_{\delta,j}$ for $1\leq j \leq d_y,$ where $C_{\delta,j}$ is a constant that depends on $\M$, $\delta$ and the observation point $x_j$.

Our discretized forward model $F:\R_M^n\to \R^{d_y}$ is then given by
\begin{align}\label{eq:graph forward model}
    F:=\mathcal{O}_nG_n.
\end{align}
\end{example}

\section{Error Analysis}\label{sec:unifiederroranalysis}
This section analyzes the computational framework introduced in Section \ref{sec: formulation}. We start in Subsection \ref{ssec:generalerroranalysis} by establishing error bounds on the infinite-dimensional posterior mean and covariance based on a general assumption. Next, we verify this assumption in the function space setting of Subsection \ref{ssec:functionspacesetting}, studying finite element discretizations (Subsection \ref{ssec:errorFEM}) and graph-based methods (Subsection \ref{ssec:errorgraph}). 
\subsection{Error Analysis: General Computational Framework}\label{ssec:generalerroranalysis}
The main results of this subsection, Theorems \ref{thm:posterior mean Theorem} and Theorem \ref{thm:posterior covariance Theorem} below, quantify the errors in the mean and covariance approximations 
\begin{align}\label{eq:meanErr}
    \eps_m&=\|m_{\text{post}}-P^*\vec{m}_{\text{post}}\|_{\H},\\
    \label{eq:covErr}
    \eps_C&=\|\Cm_{\text{post}}-P^*C_{\text{post}}P\|_{op},
\end{align}
where we recall that $P: \H \to \R_M^n$ is the discretization map defined in \eqref{eq:discretizationmap} and $P^*: \R_M^n \to \V \subset \H$ is the extension map defined in \eqref{eq:extensionmap}. Our bounds on $\eps_m$ and $\eps_C$ will rely on a general assumption on (i) the closeness of the covariance operator $\Cm_0: \H \to \H$ and its discrete approximation $C_0: \R^n_M \to \R^n_M,$ encoded in an assumption on the operator norm $\| \Cm_0 - P^* C_0 P \|_{op};$ (ii) the closeness of the forward model $\F: \H \to \R^{d_y}$ and its approximation $F: \R^n_M\to \R^{d_y},$ encoded in an assumption on the operator norm $\|\F - FP\|_{op};$ and (iii) the closeness of the finite-dimensional space $\V$ and the infinite-dimensional space $\H,$ encoded in an assumption that the orthogonal projection operator $P^*P: \H \to \V$ is close to the identity.

\begin{assumption}\label{assumption:errors}
The following holds:
\begin{enumerate}[leftmargin=.3in]
    \item[(i)] \emph{(Error in Approximation of Prior Covariance.)} There is a constant $r_1 \in \R$ and a function $\psi_1: \N \to \R$ with $\lim_{n\to \infty} \psi_1(n)=0$ such that $\|\Cm_0-P^*C_0P\|_{op}\leq r_1\psi_1(n).$ 
    \item[(ii)] \emph{(Error in Approximation of Forward Model.)} There is a constant $r_2 \in \R$ and a function $\psi_2: \N \to \R$ with $\lim_{n\to \infty} \psi_2(n)=0$ such that $\|\F-FP\|_{op}\leq r_2 \psi_2(n).$ 
    \item[(iii)] \emph{(Error in Orthogonal Projection.)} 
        There is a Hilbert space $\H'$ continuously embedded in $\H,$ a constant $r_3 \in \R$ and a function $\psi_3 :\N \to \R$ with $\lim_{n\to \infty} \psi_3(n)=0$ such that, for every $u \in \H',$  $\|(I-P^*P)u\|_{\H}\leq r_3\psi_3(n)\|u\|_{\H'}.$  
\end{enumerate}
\end{assumption}
   In Subsections \ref{ssec:errorFEM} and \ref{ssec:errorgraph}, we will verify Assumption \ref{assumption:errors} for the important examples of finite element and graph discretization spaces, Mat\'ern-type prior covariance operators, and deconvolution forward models.  In these examples, we will take the space $\H'$ in Assumption \ref{assumption:errors} (iii) to be an appropriate Sobolev space contained in the Cameron-Martin space $\text{Im}(\Cm_0^{1/2})$.

To establish upper bounds on the mean and covariance approximation errors \eqref{eq:meanErr} and \eqref{eq:covErr}, we will first prove a lemma similar to those in \cite{ghattas2022non} and \cite{kwiatkowski2015convergence}. Following the approach in these papers, we introduce the Kalman gain operator
\begin{align}\label{eq:kalman}
    \msK:  S_+ \times B(\mathcal{H},\R^{d_y}) \to B(\R^{d_y},\mathcal{H}), \quad \quad  \msK(\Cm, \F)=\Cm\F^*(\F\Cm\F^*+\Gamma)^{-1}.
\end{align}
Unlike \cite{ghattas2022non,kwiatkowski2015convergence}, we view the Kalman gain operator as a function of both the forward map and the covariance, not just the covariance. 
The following lemma shows that the Kalman gain operator is pointwise continuous.
\begin{lemma}[Continuity of Kalman Gain Update]
\label{lemma:continuity}
Let $\msK$ be the Kalman gain operator defined in \eqref{eq:kalman}. Let $\Cm_1,\Cm_2\in S_+$, let $\F_1,\F_2\in B(\mathcal{H},\R^{d_y}),$ and let $\Gamma\in S_{++}^{d_y}.$ The following holds: 
\begin{align}
    \|\msK(\Cm_1,\F_1)-\msK(\Cm_2,\F_2)\|_{op}\leq c_1\|\Cm_1-\Cm_2\|_{op} + c_2\|\F_1-\F_2\|_{op},
\end{align}
where
\begin{align}\label{eq:constantsc1c2}
\begin{split}
c_1&=\|\Gamma^{-1}\|_{op}\|\F_2\|_{op}+\|\Gamma^{-1}\|_{op}^2\|\Cm_1\|_{op}\|\F_1\|_{op}\|\F_2\|_{op}^2, \\ 
c_2&=\|\Gamma^{-1}\|_{op}\|\Cm_1\|_{op}+\|\Gamma^{-1}\|_{op}^2\|\Cm_1\|_{op}^2\|\F_1\|_{op}^2+\|\Gamma^{-1}\|_{op}\|\Cm_1\|_{op}^2\|\F_1\|_{op}\|\F_2\|_{op}.
\end{split}
\end{align}
\begin{proof}
    Applying Lemma A.8 from \cite{ghattas2022non} (which was stated for matrices, but also holds in our infinite-dimensional setting), we get that
    \begin{align*}
        \|\msK(\F_1,\Cm_1)-\msK(\F_2,\Cm_2)\|_{op} &\leq \|\Gamma^{-1}\|_{op}\|\Cm_1\F_1^*-\Cm_2\F_2^*\|_{op} \\
        &+\|\Gamma^{-1}\|_{op}^2\|\Cm_1\F_1^*\|_{op}\|\F_1\Cm_1\F_1^*-\F_2\Cm_2\F_2^*\|_{op}.
    \end{align*}
    We can then bound each of these terms with the triangle inequality:
    \begin{align*}
        \|\Cm_1\F_1^*-\Cm_2\F_2^*\|_{op}&=\|\Cm_1\F_1^*-\Cm_1\F_2^*+\Cm_1\F_2^*-\Cm_2\F_2^*\|_{op} \\
        &\leq \|\Cm_1\|_{op}\|\F_1-\F_2\|_{op}+\|\F_2\|_{op}\|\Cm_1-\Cm_2\|_{op},
    \end{align*}
    and
    \begin{align*}
        \|\F_1\Cm_1\F_1^*-\F_2\Cm_2\F_2^*\|_{op}&=\|\F_1\Cm_1\F_1^*-\F_2\Cm_1\F_1^*+\F_2\Cm_1\F_1^*-\F_2\Cm_2\F_2^*\|_{op} \\
        &\leq \|\Cm_1\F_1^*\|_{op}\|\F_1-\F_2\|_{op}+\|\F_2\|_{op}\|\Cm_1\F_1^*-\Cm_2\F_2^*\|_{op}.
    \end{align*}
    Combining the previous three displayed inequalities gives the desired result.
\end{proof}
    
\end{lemma}

\begin{theorem} [Mean Approximation Error] \label{thm:posterior mean Theorem}
    Consider the error $\eps_m$ in \eqref{eq:meanErr} between the true posterior mean in \eqref{eq:infiniteMean}  and its finite-dimensional posterior approximation in \eqref{eq:finiteMean}. Then, under Assumption \ref{assumption:errors} it holds that
    \begin{align}
        \eps_m\leq r_1'\psi_1(n)+r_2'\psi_2(n)+r_3'\psi_3(n),
    \end{align}
    where
    \begin{align*}
        r_1'  &= 2c_1r_1\|y-\F m_0\|_2 , \\
            r_2'  &=  2c_2r_2\|y-\F m_0\|_2 + r_2\|m_0 \|_\H \|\msK(\Cm_0,\F)\|_{op} ,  \\
                r_3' &= r_3 \|m_0\|_{\H'},
    \end{align*}
        and $c_1,$ $c_2$ are defined in \eqref{eq:constantsc1c2}.
    \begin{proof}
    By the definition of $\eps_m$ in \eqref{eq:meanErr} and an application of the triangle inequality,
        \begin{align}\label{eq:auxboundepsm}
        \begin{split}
            \eps_m \le 
            \|m_0-P^*\vec{m}_0\|_{\H} &+\|\msK(\Cm_0,\F)-\msK(P^*C_0P,FP)\|_{op}\|y-F\vec{m}_0\|_2  \\
          &+\|\msK(\Cm_0,\F)\|_{op}\|\F m_0-F\vec{m}_0\|_{2}.
          \end{split}
        \end{align}
        We next bound each of the terms on the right-hand side.
        By Assumption \ref{assumption:errors} (iii) and the fact that $\vec{m}_0=Pm_0$, we have that
        \begin{align}\label{eq:aux1term}
            \|m_0-P^*\vec{m}_0\|_{\H}=\|m_0-P^*Pm_0\|_{\H}\leq r_3 \|m_0\|_{\H'} \psi_3(n).
        \end{align}
        By Lemma \ref{lemma:continuity} and Assumption \ref{assumption:errors} (i) and (ii), we have that
        \begin{align}\label{eq:aux2term}
         \begin{split}
            \|\msK(\Cm_0,\F)-\msK(P^*C_0P,FP)\|_{op} &\leq c_1\|\Cm_0-P^*C_0P\|_{op} + c_2\|\F-FP\|_{op}\\
            &\leq c_1r_1\psi_1(n)+c_2r_2\psi_2(n).
            \end{split}
        \end{align}
        By Assumption \ref{assumption:errors} (ii) and (iii) we have that
        \begin{align*} 
            \begin{split}
            \|y-F\vec{m}_0\| &\leq \|y-\F m_0\|+\|\F\|_{op}\|m_0-Pm_0\|_{\H}+\|\F-FP\|_{op}\|m_0\|_{\H}\\
           & \leq \|y-\F m_0\|+\|\F\|_{op}\|m_0\|_{\H'}r_3\psi_3(n)+\|m_0\|_{\H}r_2\psi_2(n).
            \end{split}
        \end{align*}
        This implies that for $n$ sufficiently large, it holds that
        \begin{align}\label{eq:aux4term}
            \|y-F\vec{m}_0\|_2\leq 2\|y-\F m_0\|_2.
        \end{align}
        Finally, by Assumption \ref{assumption:errors} (ii) we have that
        \begin{align}\label{eq:aux3term}
            \|\F m_0-F\vec{m}_0\|_{2}\leq \|\F-FP\|_{op}\|m_0\|_{\H}\leq r_2 \|m_0\|_{\H} \psi_2(n).
        \end{align}
        Plugging in the bounds \eqref{eq:aux1term}, \eqref{eq:aux2term}, \eqref{eq:aux4term}, and \eqref{eq:aux3term} into inequality \eqref{eq:auxboundepsm} gives the desired result.
    \end{proof}
\end{theorem}

\begin{theorem} [Covariance Approximation Error]\label{thm:posterior covariance Theorem}
    Consider the error $\eps_C$ in \eqref{eq:covErr} between the true posterior covariance operator in \eqref{eq:infiniteCov} and its finite-dimensional posterior approximation in \eqref{eq:finiteCov}. Then, under Assumption \ref{assumption:errors} it holds that
    \begin{align}
        \eps_C\leq r_1'\psi_1(n)+r_2'\psi_2(n),
    \end{align}
  where
  \begin{align*}
      r_1' &= r_1 + 2c_1r_1 \|\F\|_{op}\|\Cm_0\|_{op} + 2r_1 \|\F\|_{op}\|\msK(\Cm_0,\F)\|_{op}, \\
      r_2' &= 2c_2r_2 \|\F\Cm_0\|_{op} + r_2\|\Cm_0\|_{op}\|\msK(\Cm_0,\F)\|_{op},
  \end{align*}
  and $c_1,$ $c_2$ are defined in \eqref{eq:constantsc1c2}.
\begin{proof}
By the definition of $\eps_C$ in \eqref{eq:covErr} and an application of the triangle inequality,
    \begin{align}\label{eq:auxboundepsc}
    \begin{split}
        \eps_C  
        \leq \|\Cm_0-P^*C_0P\|_{op}   &+   \|\msK(\Cm_0,\F)-\msK(P^*C_0P,FP)\|_{op}\|FC_0P\|_{op}\\
        & +\|\msK(\Cm_0,\F)\|_{op}\|\F\Cm_0-FC_0P\|_{op}.
        \end{split}
    \end{align}
    By Assumption \ref{assumption:errors} (i) we have that
    \begin{align}\label{eq:aux1epsc}
        \|\Cm_0-PC_0P\|_{op}\leq r_1\psi_1(n).
    \end{align}
    As in the previous proof, applying Lemma \ref{lemma:continuity} and using Assumption \ref{assumption:errors}  (i) and (ii) gives
     \begin{align}\label{eq:aux2epsc}
            \|\msK(\Cm_0, \F)-\msK(P^*C_0P,FP)\|_{op} 
            \leq c_1r_1\psi_1(n)+c_2r_2\psi_2(n).
        \end{align}
        The triangle inequality along with Assumption \ref{assumption:errors} (i) and (ii) yield the bound
        \begin{align}\label{eq:aux3epsc}
        \begin{split}
            \|\F\Cm_0-FC_0P\|_{op} &\leq \| FP\|_{op} \|\Cm_0-P^*C_0P\|_{op} + \|\F-FP\|_{op}\|\Cm_0\|_{op} \\
            &\leq r_1 \|FP\|_{op}\psi_1(n) + r_2 \|\Cm_0\|_{op}\psi_2(n).
            \end{split}
        \end{align}
        Assumption \ref{assumption:errors} (i) and (ii) guarantee that, for sufficiently large $n$,
        \begin{align}\label{eq:aux4epsc}
            \|FP\|_{op}\leq 2\|\F\|_{op}, \text{ and } \|FC_0P\|_{op}\leq 2\|\F\|_{op}\|\Cm_0\|_{op}.   
        \end{align}
     Plugging in the bounds \eqref{eq:aux1epsc}, \eqref{eq:aux2epsc}, \eqref{eq:aux3epsc}, and \eqref{eq:aux4epsc} into \eqref{eq:auxboundepsc} gives the desired result. 
\end{proof}
    
\end{theorem}

\subsection{Error Analysis: Finite Element Discretizations}\label{ssec:errorFEM}
The following result applies Theorems \ref{thm:posterior mean Theorem} and \ref{thm:posterior covariance Theorem} to quantify the errors in the finite element approximations to the posterior mean and covariance operators in the setting considered in Subsection \ref{sec:FEM discretization}.

\begin{theorem}[Finite Element Posterior Mean and Covariance Approximation]
\label{thm:FEM Posterior Theorem}
    Consider the discretization proposed in Subsection \ref{sec:FEM discretization} with a finite element space $\V_h$ of linear Lagrange basis vectors and time discretization step of $\Delta t.$  Assume the prior mean function is chosen to be in the Sobolev space $H^s(\D)$. Let $\Theta(x)\in C^1(\bar{\D}),$ $b(x)\in L^{\infty}(\D)$, and assume that both functions are (almost surely) bounded below by positive constants. The errors in the mean and covariance approximations are bounded by
    \begin{align*}
        \|m_{\emph{post}}(x)-P^*\vec{m}_{\emph{post}}\|_{L^2(\D)}&\lesssim h^{\min  \{2,s\}}+\Delta t^2, \\
        \|\Cm_{\emph{post}}-P^*C_0P\|_{op}&\lesssim h^{2}+\Delta t^2.
    \end{align*}
\end{theorem}
This result follows from Theorems \ref{thm:posterior mean Theorem} and \ref{thm:posterior covariance Theorem} upon verifying that Assumption \ref{assumption:errors} hold for the finite element discretizations of our model Bayesian inverse problem. Verifying these assumptions will be the focus of the following three subsections. In so doing, we will view the $\psi_i$ in Assumption \ref{assumption:errors} as functions of the mesh width  $h,$ noting that this quantity determines the dimension $n$ of the discretization space. For simplicity, we have restricted our discussion to linear Lagrange basis vectors. As our general theory suggests, higher degree polynomials could be used in the finite element approximations to yield faster convergence rates. 
 
\subsubsection{Finite Element Approximation of Matérn-Type Prior Covariance}\label{ssec:errorcovarianceFEM}
Here we show that the finite element approximation to the covariance operator given in Subsection \ref{sec:FEM discretization} satisfies Assumption \ref{assumption:errors}  (i).  
    \begin{theorem}[Operator-Norm Error for Finite Element Prior Covariance Approximation] \label{th:opNormPriorCov}
        Let $\Cm_0:L^2(\D)\to L^2(\D)$ be the prior covariance operator defined in Section \ref{sec: formulation}, and let $A^{-\alpha}:\R^n_M\to \R^n_M$ be the finite-dimensional approximation defined in Subsection \ref{sec:FEM discretization} corresponding to a finite element space $\V_h$ with piecewise linear basis functions.  Let $\Theta(x)\in C^1(\bar{\D}),$ $b(x)\in L^{\infty}(\D)$, and assume that both functions are (almost surely) bounded below by positive constants\color{black}.  Then, there exists a constant $c$  independent of  $h$ such that
        \begin{align}\label{eqn:FEM op norm prior bound}
            \|\Cm_0-P^*A^{-\alpha}P\|_{op}\leq ch^{2}.
        \end{align}
    \end{theorem}
The proof can be found in Appendix \ref{AppendixFEMProofs}. This shows that the finite-dimensional prior approximation proposed in Subsection \ref{sec:FEM discretization} satisfies Assumption \ref{assumption:errors}  (i) with $\psi_1(h)=h^2.$ 

\subsubsection{Finite Element Approximation of Heat Forward Model}\label{ssec:errorforwardFEM}
We now show that the discretized finite element forward map given in \eqref{eq:discrete FEM forward map} satisfies Assumption \ref{assumption:errors} (ii). 
\begin{theorem}[Operator-Norm Error for Finite Element Forward Map Approximation]\label{thm:forwardFEM}
Let $\F:L^2(\D)\to \R^{d_y}$ be the forward model defined in \eqref{eq:forward model}, and let $F:\R^n_M \to \R^{d_y}$ be the finite-dimensional approximation defined in \eqref{eq:discrete FEM forward map} corresponding to a finite element space $\V_h$ consisting of linear Lagrange basis vectors and a time-discretization step of $\Delta t.$ Then, there exist constants $c_1$ and $c_2$  independent of $h$ and $\Delta t$ such that
\begin{align}
    \|\F-FP\|_{op}\leq c_1h^{2}+c_2\Delta t^2.
\end{align}
\end{theorem}
The proof is given in Appendix \ref{AppendixFEMProofs}. We see that Assumption \ref{assumption:errors} (ii) holds for the Crank-Nicolson discretization with $\psi_{2}(h)=h^2+\Delta t^2.$
\subsubsection{Finite Element Orthogonal Projection}\label{ssec:errorprojectionFEM}
We first verify that the orthogonal projection of the prior mean function onto the finite element subspace $\V_h$ converges as in Assumption \ref{assumption:errors} (iii)  as the mesh is refined. We let $m_0(x)\in H^s(\D)$. Since we require that the mean function lies in the Cameron-Martin space $E= \text{Im}(\Cm_0^{1/2})$, we must have that $s\geq \alpha$. We assume that $\V_h$ consists of piecewise linear basis functions. Since $s\geq \alpha>d/2,$ the Sobolev embedding theorem guarantees that functions in $H^s(\D)$ are continuous. Consequently, the prior mean function can be well approximated by piecewise linear interpolants and, in turn, by its orthogonal projection onto $\V_h$. Assuming that the finite element space $\V_h$ is given by a quasiuniform mesh refinement of $\D$ with linear Lagrange basis functions, \cite[Theorem 6.8]{oden2012introduction} guarantees that, for any $m_0\in H^{s}(\D),$ there exists an element $U\in  \V_h$ such that $$\|m_0-U\|_{L^2(\D)}\leq ch^{\min \{2,s\}}\|m_0\|_{H^s(\D)}.$$
Since the orthogonal projection is defined to minimize the $L^2(\D)$ error over all functions in $\V_h$, we have that, for any $m_0\in H^{s}(\D),$
\begin{align}\label{eq:FEM projection bound}
    \|m_0-P^*Pm_0\|_{L^2(\D)}\leq \|m_0-U\|_{L^2(\D)}\leq ch^{\min\{2,s\}}\|m_0\|_{H^{s}(\D)}.
\end{align}
 We see that Assumption \ref{assumption:errors} (iii) holds for the finite element discretization with $\psi_3(h)=h^{\min\{2,s\}}$  and $\H'=H^s(\D).$

We have now verified that Assumption \ref{assumption:errors} holds for the finite element discretization, proving Theorem \ref{thm:FEM Posterior Theorem}.

\subsection{Error Analysis: Graph-Based Approximations}
\label{ssec:errorgraph}
The following result applies Theorems \ref{thm:posterior mean Theorem} and \ref{thm:posterior covariance Theorem} to quantify the error in the graph-based approximations to posterior mean and covariance operator in the setting considered in Subsection \ref{sec:graph discretization}.
\begin{theorem}[Graph-Based Posterior Mean and Covariance Approximation]
\label{thm:Graph Posterior Theorem}
Consider the discretization proposed in Subsection \ref{sec:graph discretization}. Assume we are in a realization in which the conclusion of Proposition \ref{thm:transportMaps} holds and we have that $\alpha>(5d+1)/4$.  Let $b(x)$ be Lipschitz, $\Theta(x)\in C^1(\M)$, and suppose that both are bounded below by positive constants. Then, if we take the scaling $h_n\asymp \sqrt{\frac{(\log n)^{c_d}}{n^{1/d}}},$
the errors in the posterior mean and covariance approximations are bounded by
\begin{align*}
\begin{split}
    \|m_{\emph{post}}(x)-P^*\vec{m}_{\emph{post}}\|_{L^2(\M)}\lesssim (\log n)^{\frac{c_d}{4}}n^{-\frac{1}{4d}},\\
    \|\Cm_{\emph{post}}-P^*C_{\emph{post}}P\|_{op}\lesssim (\log n)^{\frac{c_d}{4}}n^{-\frac{1}{4d}},
    \end{split}
\end{align*}
 where $c_d=3/4$ if $d=2$ and $c_d=1/d$ otherwise.
\end{theorem}
This result follows from Theorems \ref{thm:posterior mean Theorem} and \ref{thm:posterior covariance Theorem} upon verifying that the graph-based discretization satisfies Assumption \ref{assumption:errors} for our model problem. Verifying these assumptions will be the focus of the following subsections. In all that follows, we assume to be in a realization where the conclusion of Proposition \ref{thm:transportMaps} holds. 

\subsubsection{Graph-Based Approximation of Matérn-Type Prior Covariance}\label{ssec:errorcovariancegraph}
The following theorem shows that the covariance operator constructed in Section \ref{sec:graph discretization} satisfies Assumption \ref{assumption:errors} (ii). The proof can be found in Appendix \ref{AppendixGraphProofs}. 

\begin{theorem}[Operator-Norm Error for Graph-Based Prior Covariance Approximation]
    \label{thm:opNormGraphCov}
    Let $b(x)$ be Lipschitz, $\Theta(x)\in C^1(\M)$, and suppose that both are bounded below by positive constants. Let $\alpha>(5d+1)/4$, and take  
$h_n \asymp \sqrt{\frac{(\log n)^{c_d}}{n^{1/d}}}$.
   Then,
    \begin{align*}
        \|\Cm_0-P^*(\Delta_n^{\Theta}+B_n)^{-\alpha}P\|_{op}\lesssim (\log n)^{\frac{c_d}{4}}n^{-\frac{1}{4d}},
    \end{align*} 
    where $c_d=3/4$ if $d=2$ and $c_d=1/d$ otherwise.
\end{theorem}

\begin{remark}
The proof of Theorem \ref{thm:opNormGraphCov} in Appendix \ref{AppendixGraphProofs} follows that of Theorem D.1 in \cite{sanz2022spde}. However, since we are working with the covariance operator (as opposed to the square root of the covariance operator as is the case when sampling the Mat\'ern fields), we get convergence rates for $\alpha>(5d+1)/4$ as opposed to $\alpha>(5d+1)/2$. The rates of convergence are still the same, since the error is still dominated by the error in approximating the eigenfunctions.
\end{remark}
We have shown that the graph-based finite-dimensional approximation to the prior covariance proposed in Subsection \ref{sec:graph discretization} satisfies Assumption \ref{assumption:errors} (i) with $\psi_1(n)=(\log n)^{\frac{c_d}{4}}n^{-\frac{1}{4d}}$, where $c_d=3/4$ if $d=2$ and $c_d=1/d$ otherwise.
\subsubsection{Graph-Based Approximation of Heat Forward Model}\label{ssec:errorforwardgraph}
We show that the graph-based approximation of the heat forward model satisfies with the surrogate observation operator satisfies Assumption \ref{assumption:errors} (iii). The proof can be found in Appendix \ref{AppendixGraphProofs}.

\begin{theorem}[Operator-Norm Error for Graph-Based Forward Map Approximation]\label{thm:forwardgraph}
     Let $F=\mathcal{O}_nG_n$ be the forward model defined in Subsection \ref{sec:graph discretization}. Take $h_n \asymp \sqrt{\frac{(\log n)^{c_d}}{n^{1/d}}}.$ Then,
    \begin{align}
        \|\F-P^*FP\|_{op}\lesssim (\log n)^{\frac{c_d}{4}}n^{-\frac{1}{4d}}.
    \end{align}
\end{theorem}
We have shown that the graph-based finite-dimensional approximation to the forward model proposed in Subsection \ref{sec:graph discretization} satisfies Assumption \ref{assumption:errors} (ii) with $\psi_2(n)=(\log n)^{\frac{c_d}{4}}n^{-\frac{1}{4d}}$, where $c_d=3/4$ if $d=2$ and $c_d=1/d$ otherwise.
\subsubsection{Graph-Based Orthogonal Projection}\label{ssec:errorprojectiongraph}
Finally, we verify that the orthogonal projection of the prior mean function onto the subspace $\V=\text{span}\{\mathbb{1}_{U_i}(x)\}_{i=1}^n$ converges as in Assumption  \ref{assumption:errors} (iii). Since the mean function must lie in the Cameron-Martin space $E=\text{Im}(\Cm_0^{1/2})$, we necessarily have that $m_0(x)\in H^1(\M).$ The following result from \cite[Lemma 12]{garcia2020error} quantifies the convergence rate of the projected mean function:
\begin{lemma}[Lemma 12 in \cite{garcia2020error}]
    There exists a constant $c_{\M}$ independent of $n$ such that, for every $m_0\in H^1(\M)$, we have
    \begin{align}
        \|m_0-P^*Pm_0\|_{L^2(\M)}\leq c_{\M}\eps_n \|m_0\|_{H^1(\M)}\lesssim \frac{(\log n)^{c_d}}{n^{1/d}}\|m_0\|_{H^1(\M)}.
    \end{align}
\end{lemma}
We see that Assumption \ref{assumption:errors} (iii) holds for the graph-based discretization with $\psi_3(n)=\frac{(\log n)^{c_d}}{n^{1/d}}$  and $\H'=H^1(\M)$. For large $n$, this term will be dominated by $\psi_1(n)$ and $\psi_2(n).$

We have now verified that Assumption \ref{assumption:errors} holds for the graph-based discretization, proving Theorem \ref{thm:Graph Posterior Theorem}.

\section{Sample Size Requirements for Ensemble Kalman Updates}\label{sec:ensembleKalmanupdates}
In this section, we formulate the ensemble Kalman update in weighted inner product space and show that the effective dimension of the discretized problem is controlled by the effective dimension of the continuum problem, which is finite.

\subsection{Ensemble Approximation to Finite-Dimensional Posterior}
To implement an ensemble Kalman approximation to the finite-dimensional Gaussian posterior given by \eqref{eq:finitePost}, we need to draw samples from the Gaussian prior distribution, $\Nc(\vec{m}_0,C_0)$ given by \eqref{eq:finitePrior}. To do so, we draw $\xi^{(j)}\sim \Nc(0,I),$ and compute samples 
\begin{align}\label{eq:priorDraws}
    u^{(j)}=\vec{m}_0+L\xi^{(j)}, \quad u^{(j)}\sim \mathcal{N}(\vec{m}_0,C_0),
\end{align}
for $1\leq j \leq J$, 
where $L$ is a linear map from $\R^n\to \R^n_M$ such that $C_0=LL^{\sharp}=LL^TM.$ Recall that given samples $u^{(1)},\ldots,u^{(J)}\sim \Nc(0,\Cm)$ from a Gaussian measure in a Hilbert space $\H$, the sample covariance operator is defined as the operator $\widehat{\hspace{0pt} \Cm}:\H\to \H$
$$
\widehat{\hspace{0pt} \Cm}u:=\frac{1}{J-1}\sum_{j=1}^J\langle u^{(j)},u\rangle_{\H}u^{(j)}, \quad u\in \H.
$$
As such, the sample covariance operator on $\R_M^n$ is given by 
\begin{align*}
    \widehat{C}_0=\frac{1}{J-1}\sum_{j=1}^J(u^{(j)}-\widehat{m}_0)(u^{(j)}-\widehat{m}_0)^TM,
\end{align*}
where $\widehat{m}_0=\frac{1}{J}\sum_{j=1}^J u^{(j)}$ is the sample mean. The perturbed observation ensemble Kalman update transforms each sample from the prior ensemble via
\begin{align}\label{eq:POUpdate}
v^{(j)}=u^{(j)}+\widehat{C}_0F^{\natural}\left( F\widehat{C}_0F^{\natural}+\Gamma\right)^{-1}\left(y-Fu^{(j)}+\eta^{(j)} \right), \quad \eta^{(j)}\iid \Nc(0,\Gamma).
\end{align}
We then output the sample mean and covariance of the transformed ensemble, given by
\begin{align}
\label{eq:enspostmean}
\widehat{m}_{\text{post}}=\frac{1}{J}\sum_{j=1}^Jv^{(j)}, \text{ and } \widehat{C}_{\text{post}}&=\frac{1}{J-1}\sum_{J=1}^{J}(v^{(j)}-\widehat{m}_{\text{post}})(v^{(j)}-\widehat{m}_{\text{post}})^TM.
\end{align}

\subsection{Ensemble Kalman Approximation: Error Analysis}
We now want to derive non-asymptotic expectation bounds for the error in the ensemble approximation to the infinite-dimensional posterior. We consider the errors
\begin{align} 
    \widehat{\eps}_m &=\Expect\|m_{\text{post}}-P^*\widehat{m}_{\text{post}}\|_{\H} ,\label{eq:ensMeanErr} \\
\hat{\eps}_C &=\Expect\|\Cm_{\text{post}}-P^*\hat{C}_{\text{post}}P\|_{op}.   \label{eq:ensCovErr}
\end{align}
Controlling these errors will amount to controlling the error between the ensemble approximations and the finite-dimensional approximations using the theory from \cite{ghattas2022non} and then applying our results from Section \ref{sec:unifiederroranalysis}. Since the dimension of the discretized covariance matrices increases as the mesh is refined, we will require bounds that do not have an explicit dependence on the dimension. To do so, we define the \textit{effective dimension} of a mapping $C_0:\R_M^n\to \R_M^n$ to be 
\begin{align}\label{eq:effectiveDim}
    r_M(C_0):=\frac{\text{Tr}(C_0)}{\|C_0\|_{op}},
\end{align}
where $\|\cdot\|_{op}$ is the operator norm from $\R_M^n$ to $\R_M^n.$ When the eigenvalues of $C_0$ decay rapidly, the effective dimension $r_M(C_0)$ is  a better measure of dimension than the nominal state dimension $n$ \cite{tropp2015introduction}. We will show in Subsection \ref{sec:EffectiveDimensionBound} that for the finite element discretization, the effective dimension of the discretized operator is controlled by the effective dimension of the continuum operator, defined to be $r(\Cm_0)=\frac{\text{Tr}(\Cm_0)}{\|\Cm_0\|_{op}}$, where $\|\cdot\|_{op}$ denotes the operator norm from $\H$ to $\H$. Since the prior covariance operator must be chosen to be bounded and trace-class, this quantity is necessarily finite. As such, the error bounds that we now derive for the ensemble approximation will not degenerate as the mesh is refined, despite the state dimension increasing arbitrarily.

\begin{theorem}[Mean Ensemble Approximation Error]\label{thm:postMeanEnsTheorem}
    Let $\widehat{\eps}_m$ be the expected error between the infinite-dimensional posterior mean and the ensemble posterior mean with an ensemble of size $J$, as defined in \eqref{eq:ensMeanErr}. Assume that the finite-dimensional approximations satisfy Assumption \ref{assumption:errors}, that for $n$ sufficiently large $r_M(C_0)\leq c r(\Cm_0)$ for some constant $c$ independent of $n$, and that $J\geq r_M(C_0)$. Then, it holds that 
    \begin{align}
        \widehat{\eps}_m\leq r_1'\psi_1(n)+r_2'\psi_2(n)+r'_3\psi_3(n)+r_4'\frac{1}{\sqrt{J}},
    \end{align}
    where the constants $r_i'$ are independent of  $n$ and $J.$
\begin{proof}
By the triangle inequality, we decompose the expected  error as
\begin{align*}
    \Expect\|m_{\text{post}}(x)-P^*\widehat{m}_{\text{post}(x)}\|_{L^2(\D)}\leq \Expect\|m_{\text{post}}(x)-P^*\vec{m}_{\text{post}}(x)\|_{L^2}+\Expect\|\widehat{m}_{\text{post}}-\vec{m}_{\text{post}}\|_M.
\end{align*}
The first term on the right-hand side is deterministic and is bounded by Theorem \ref{thm:posterior mean Theorem}:
\begin{align*}
    \Expect\|m_{\text{post}}(x)-P^*\vec{m}_{\text{post}}(x)\|_{L^2}=\|m_{\text{post}}(x)-P^*\vec{m}_{\text{post}}(x)\|_{L^2}\lesssim r_1'\psi_1(n)+r_2'\psi_2(n)+r_3'\psi_3(n).
\end{align*}
For the second term on the right-hand side, we recall \cite[Theorem 3.3]{ghattas2022non}, which guarantees the following:
\begin{align}
    \Expect\|\widehat{m}_{\text{post}}-\vec{m}_{\text{post}}\|_M\lesssim c_n'\sqrt{\frac{r_M(C_0)}{J}}+c_n''\sqrt{\frac{r_2(\Gamma)}{J}},
\end{align}
where
$c_n'=\left(\|C_0\|_{op}^{1/2}\vee \|C_0\|_{op}^2 \right)\left(\|F\|_{op}\vee \|F\|_{op}^4 \right)\left(\|\Gamma^{-1}\|_{op}\vee \|\Gamma^{-1}\|_{op}^2 \right)\left(1\vee \|y-F\vec{m}_0\|_2 \right)$, and $c_n''=\|F\|_{op}\|\Gamma^{-1}\|_{op}\|\Gamma\|_{op}^{1/2}\|C_0\|_{op}$. The quantity $r_2(\Gamma)$ is the effective dimension of $\Gamma$ in the usual Euclidean inner product space. The quantities $c_n'$ and $c_n''$ depend on $n;$ however, under Assumption \ref{assumption:errors} each of the discretized operators appearing in the constants can be made arbitrarily close to their continuum counterparts when $n$ is sufficiently large. In particular, we can take $n$ to be large enough such that the following two bounds hold:
\begin{align*}
c_n' &\leq c'=2\left(\|\Cm_0\|_{op}^{1/2}\vee \|\Cm_0\|_{op}^2 \right)\left(\|\F\|_{op}\vee \|\F\|_{op}^4 \right)\left(\|\Gamma^{-1}\|_{op}\vee \|\Gamma^{-1}\|_{op}^2 \right)\left(1\vee \|y-\F m_0\|_2 \right), \\
    c_n'' &\leq c''=2\|\F\|_{op}\|\Gamma^{-1}\|_{op}\|\Gamma\|_{op}^{1/2}\|\Cm_0\|_{op}.
\end{align*}
From these bounds combined with our assumption that $r_M(C_0)\leq cr(\Cm_0)$, we take $r_4'=\sqrt{c}c'\sqrt{r(\Cm_0)}+c''\sqrt{r(\Gamma)}$, and the result follows.
\end{proof}
\end{theorem}

\begin{theorem}[Covariance Ensemble Approximation Error]\label{thm:postCovEnsTheorem}
    Let $\widehat{\eps}_C$ be the expected error between the infinite-dimensional posterior covariance and the ensemble posterior covariance with an ensemble of size $J$, as defined in \eqref{eq:ensMeanErr}. Assume that the finite-dimensional approximations satisfy Assumption \ref{assumption:errors}, that for $n$ sufficiently large $r_M(C_0)\leq c r(\Cm_0)$ for some constant $c$ independent of $n$, and that $J\geq r_M(C_0)$. Then, it holds that 
    \begin{align}
        \widehat{\eps}_C\leq r_1'\psi_1(n)+r_2'\psi_2(n)+r_3'\frac{1}{\sqrt{J}},
    \end{align}
    where the constants $r_i'$ are independent of the discretization dimension $n$ and the ensemble size $J.$
\begin{proof}
The proof proceeds similarly to that of Theorem \ref{thm:postMeanEnsTheorem}. We decompose the error as
\begin{align}
    \Expect\|\Cm_{\text{post}}-P^*\widehat{C}_{\text{post}}P\|_{op}\leq \Expect\|\Cm_{\text{post}}-P^*C_{\text{post}}P\|_{op}+\Expect\|C_{\text{post}}-\widehat{C}_{\text{post}}\|_{op}.
\end{align}
The first term on the right-hand side is deterministic and is controlled by Theorem \ref{thm:posterior covariance Theorem}:
\begin{align*}
    \Expect\|\Cm_{\text{post}}-P^*C_{\text{post}}P\|_{op}=\|\Cm_{\text{post}}-P^*C_{\text{post}}P\|_{op}\lesssim r_1'\psi_1(n)+r_2'\psi_2(n).
\end{align*}
For the second term, we use \cite[Theorem 3.5]{ghattas2022non}, which guarantees that
\begin{align}
    \Expect\|\widehat{C}_{\text{post}}-C_{\text{post}}\|_{op}\lesssim c_n'\sqrt{\frac{r_M(C_0)}{J}}+c_n''\left(\sqrt{\frac{r_M(C_0)}{J}}\vee \sqrt{\frac{r_2(\Gamma)}{J}} \right),
\end{align}
where 
\begin{align*}c_n' &=\left( \|C_0\|_{op}\vee \|C_0\|_{op}^3\right)\left(\|F\|_{op}^2\vee\|F\|_{op}^4 \right)\left( \|\Gamma^{-1}\|_{op}\vee \|\Gamma^{-1}\|_{op}^2\right), \\
    c_n'' &=\left( \|F\|_{op}\vee \|F\|_{op}^3\right)\left(\Gamma^{-1}\|\vee\|\Gamma^{-1}\|^2 \right)\left( \|C_0\|_{op}\vee \|\Gamma\|_{op}\right)\left(\|C_0\|_{op}\vee \|C_0\|^2 \right).
\end{align*}
Again, the quantities $c_n'$ and $c_n''$ depend on $n$, however under Assumption \ref{assumption:errors} each of the discretized operators appearing in the constants can be made arbitrarily close to their continuum counterparts when $n$ is sufficiently large. In particular, we take $n$ to be large enough such that the following two bounds hold:
\begin{align*}
    c_n' &\leq c'=2\left( \|\Cm_0\|_{op}\vee \|\Cm_0\|_{op}^3\right)\left(\|\F\|_{op}^2\vee\|\F\|_{op}^4 \right)\left( \|\Gamma^{-1}\|_{op}\vee \|\Gamma^{-1}\|_{op}^2\right), \\
    c_n'' &\leq c''=2\left( \|\F\|_{op}\vee \|\F\|_{op}^3\right)\left(\Gamma^{-1}\|\vee\|\Gamma^{-1}\|^2 \right)\left( \|\Cm_0\|_{op}\vee \|\Gamma\|_{op}\right)\left(\|\Cm_0\|_{op}\vee \|\Cm_0\|^2 \right).
\end{align*}
With these bounds and our assumption that $r_M(C_0)\leq c r(\Cm_0)$, we take $r_3=\sqrt{c}c'\sqrt{r(\Cm_0)}+c''\left(\sqrt{c}\sqrt{r(\Cm_0)}\vee r_2(\Gamma) \right)$ and we are done.
\end{proof}
\end{theorem}
The results in \cite{ghattas2022non} also prove high probability bounds for the error in the posterior means and covariances, which are stronger than the expectation bounds used here.  High probability bounds  can be  just as easily obtained within our computational framework, but we elect to present the expectation bounds for simplicity of exposition.

\subsection{Effective Dimension of Prior 
Covariance}\label{sec:EffectiveDimensionBound}
The results in the previous subsection relied on the assumption that the effective dimension of the discretized covariance operator in the weighted inner product space is bounded above independently of the discretization level for sufficiently large $n.$ We will now show that this assumption does indeed hold for the finite element covariance discretization given in Subsection \ref{sec:FEM discretization}. To get an upper bound for \ref{eq:effectiveDim}, we need an upper bound for $\text{Tr}(C_0)$ and a lower bound for $\|C_0\|_M$. The following result uses classical eigenvalue approximation results and our operator-norm convergence result to derive such a bound.
\begin{theorem}[Effective Dimension Upper Bound]\label{thm:effDimUB}
    Let the assumptions of Theorem \ref{th:opNormPriorCov} hold. Then, for $h$ sufficiently small, there exists a constant $\tau$ independent of $h$ such that
    \begin{align}\label{eq:efdimbound}
        r_M(C_0)\leq \frac{\emph{Tr}(\Cm_0)}{\|\Cm_0\|_{op}-\tau h^2}.
    \end{align}
    Consequently, for any constant $c>1$ taking $h\leq \sqrt{\tau(1-\frac{1}{c})\|\Cm_0\|_{op}}$ guarantees that
    \begin{align}
        r_M(C_0)\leq cr(\Cm_0).
    \end{align}
    \begin{proof}
        By Theorem \ref{th:opNormPriorCov} we have that
        \begin{align*}
            \|\Cm_0-P^*C_0P\|_{op}\leq \tau h^2,
        \end{align*}
        for some constant $\tau$ independent of $h$. Applying the reverse triangle inequality to this inequality gives us that
        \begin{align*}
            \|P^*C_0P\|_{op}\geq \|\Cm_0\|_{op}- \tau h^2.
        \end{align*}
        Then, using the fact that $\R_M^n$ matrix operator norm coincides with the $L^2(\D)$ operator norm of any mapping in the weighted inner product space, we get that
        \begin{align}\label{eq:opNorm lower bound}
            \|C_0\|_{op}\geq\|\Cm_0\|_{op}- \tau h^2.
        \end{align}
        To upper bound the trace of $C_0$ we proceed to bound the eigenvalues of $C_0.$ Recall that $C_0=A^{-\alpha}.$ We let $\{\lambda^{(i)}_h\}_{i=1}^n$ denote the eigenvalues of $A$, and $\{\lambda^{(i)}\}_{i=1}^\infty$ denote the eigenvalues of the continuum differential operator $\mathcal{A}$, both in ascending order. To characterize these eigenvalues, we can use classical finite element eigenvalue estimates. In particular, we refer to \cite[Theorem 6.1]{strang1973analysis}, which proves that the eigenvalues of the finite element approximation overestimate the eigenvalues of the continuum differential operator. That is, for each $i=1,\ldots,n$ we have that $\lambda^{(i)}\leq \lambda_h^{(i)}$. It then follows that
        \begin{align}\label{eq:trace uppper bound}
            \text{Tr}(C_0)=\sum_{i=1}^n(\lambda_h^{(i)})^{-\alpha}\leq \sum_{i=1}^n(\lambda^{(i)})^{-\alpha}\leq \sum_{i=1}^\infty (\lambda^{(i)})^{-\alpha}= \text{Tr}(\Cm_0).
        \end{align}
        Combining \eqref{eq:opNorm lower bound} and \eqref{eq:trace uppper bound} gives \eqref{eq:efdimbound}, completing the proof. 
    \end{proof}
\end{theorem}

\begin{remark}
    In the graph-based approximation setting, one must work slightly harder to derive such an upper bound on the effective dimension as the eigenvalues of the graph Laplacian are not guaranteed to overestimate the eigenvalues of the Laplace-Beltrami operator. Further, the eigenvalue estimates in Theorem \ref{thm:graphEW} only hold for eigenvalues that satisfy $h_n\sqrt{\lambda^{(k)}}\ll 1$, which typically only holds up to some $k<n$. As such, it may be most natural to consider a truncated prior covariance that retains only the portion of the spectrum that provably approximates the continuum operator, as is done in \cite{trillos2017consistency}.
\end{remark}

\section{Convergence of MAP Estimators Under Mesh Refinement}\label{sec:MAPconvergence}
For many inverse problems of interest, the forward model is not linear. In this section, we will see how our computational framework and analysis can be leveraged to guarantee the convergence of the discretized maximum a posteriori (MAP) estimators to their continuum counterparts in nonlinear Bayesian inverse problems as the mesh is refined. In particular, our framework very naturally fits into the theory developed in \cite{ayanbayev2021convergence}, from which we will show $\Gamma$-convergence of the relevant Onsager-Machlup functions and consequently convergence of the MAP estimators (up to subsequences). Under reasonable conditions on the nonlinear forward model (see Assumption 2.7 in \cite{stuart2010inverse}), the posterior can still be characterized as a change of measure with respect to the prior as in \eqref{eq:changemeasureinfinite}:
\begin{align}\label{eq:nonlinear posterior measure}
        \frac{d\mu_{\text{post}}}{d\mu_0}(u)=\frac{1}{Z}\exp \left(-\frac{1}{2}\|y-\F(u)\|_{\Gamma^{-1}}^2 \right),
\end{align}
where $Z=\int_{\H}\exp \left(-\Phi(u) \right)d\mu_0(u).$ For a general $\F$, this posterior measure will no longer be Gaussian. As such, it is difficult to fully characterize the posterior measure. One particularly useful point summary of the posterior measure is a maximum a posteriori (MAP) estimator. In infinite dimensions, the notion of a MAP estimator was introduced in \cite{dashti2013map} as the maximizer of a small ball probability. The theory of MAP estimation in function spaces has been further refined in \cite{kretschmann2019nonparametric,lambley2022order,klebanov2023maximum,ayanbayev2021convergence,lambley2023strong}. For $u\in \H,$ we let $B^{\delta}(u)\subset \H$ be the open ball centered at $u$ with radius $\delta.$ A \textit{MAP estimator} (or \textit{strong mode}) for $\mu_{\text{post}}$ is any point $u^{\text{MAP}}\in \H$ satisfying
\begin{align}\label{eq:MAP Estimator}
    \lim_{\delta\to 0}\frac{\mu_{\text{post}}\left(B^{
    \delta}(u^{\text{MAP}}) \right)}{M_\delta}=1, \quad M_\delta=\sup_{u\in \H}\mu_{\text{post}}\bigl(B^{
    \delta}(u) \bigr).
\end{align}
It is shown in \cite{lambley2023strong} that the MAP estimators of \eqref{eq:nonlinear posterior measure} with a Gaussian prior $\mu_0=\Nc(m_0,\Cm_0)$ in a separable Hilbert space are precisely characterized by the minimizers of the Onsager-Machlup (OM) functional $I_{\text{post}}:\H\to \R$ given by
\begin{align}\label{eq:OM functional}
    I_{\text{post}}(u)=\begin{cases}
        \Phi(u)+I_0(u), \quad & \text{if } u-m_0\in \text{Im}(\Cm_0^{1/2}),\\
        +\infty, & \text{otherwise},
    \end{cases}
\end{align}
where
\begin{align*}
    \Phi(u)=\frac{1}{2}\|y-\F(u)\|_{\Gamma^{-1}}^2 \quad  \text{ and } \quad  I_0(u)=\frac{1}{2}\|\Cm_0^{-1/2}(u-m_0)\|_{\H}^2.
\end{align*}
In more general settings, the minimizers of \eqref{eq:OM functional} may not coincide with MAP estimators defined in \eqref{eq:MAP Estimator}, and a weaker notion of MAP estimator is required \cite{helin2015maximum}. The optimization problem of minimizing \eqref{eq:OM functional} is in general difficult, if not impossible, to solve analytically. The computational framework put forth in Subsection \ref{ssec:computationalframework} provides a tractable finite-dimensional optimization problem that provably approximates the infinite-dimensional problem. The discretized posterior measure in \eqref{eq:changemeasurefinite} gives the following OM functional:
\begin{align}\label{eq:projected OM functional}
    I_{\text{post}}^{(n)}(u)=\begin{cases}
        \Phi^{(n)}(u)+I_0^{(n)}(u),\quad  & \text{if } u-m\in \text{Im}((P^*C_0P)^{1/2}),\\
        +\infty, & \text{otherwise},
    \end{cases}
\end{align}
where
\begin{align*}
    \Phi^{(n)}(u)=\frac{1}{2}\|y-F(Pu)\|_{\Gamma^{-1}}^2, \quad  \text{ and } \quad  I_0^{(n)}=\frac{1}{2}\|(P^*C_0P)^{\dagger/2}(u-m_0)\|_{\H}^2.
\end{align*}
Here $F:\R^n_M\to \R^k$ is a discretized approximation to $\F$ and the operator $(P^*C_0P)^{\dagger/2}=\sum_{i=1}^n(\lambda_n^{(i)})^{-\alpha/2}P^*\psi_n^{(i)}\otimes P^*\psi_n^{(i)}$ is the Moore-Penrose pseudoinverse of $(P^*C_0P)^{1/2}.$ We remark that $u-m\in \text{Im}((P^*C_0P)^{1/2})$ if and only if $u,m\in\V$, so we can equivalently write \eqref{eq:projected OM functional} as $I_{\text{post}}^{(n)}:\R_M^n\to \R,$ with
\begin{align}\label{eq:discretized OM functional}
    I_{\text{post}}^{(n)}(\vec{u})=\frac{1}{2}\|y-F(\vec{u})\|_{\Gamma^{-1}}^2+\frac{1}{2}\|C_0^{-1/2}(\vec{u}-\vec{m}_0)\|_{M}^2,
\end{align}
which can be minimized by methods from the breadth of literature on finite-dimensional optimization problems. To apply the results in \cite{ayanbayev2021convergence}, we first review some preliminary definitions and results regarding $\Gamma$-convergence.
\begin{definition}
Let $I,I_n:\H\to \overline{\R}$. We say that $I_n$ $\Gamma$-converges to $I$, or $\Gamma$-$\lim_{n\to\infty}I_n=I$, if, for every $u\in \H$, the following two conditions hold:
\begin{enumerate}[leftmargin=.3in]
    \item[(a)] \emph{($\Gamma$-lim inf Inequality.)} For every sequence $\{u_n\}_{n=1}^\infty$ converging to $u$ in $\H$,
    \begin{align}\label{eq:lim inf inequality}
        I(u)\leq \liminf_{n\to \infty}I_n(u_n);
    \end{align}
    \item[(b)] \emph{($\Gamma$-lim sup Inequality.)} There exists a sequence $\{u_n\}_{n=1}^\infty$ converging to $u$ such that 
    \begin{align}\label{eq:lim sup inequality}
        I(u)\geq \limsup_{n\to \infty} I_n(u_n).
    \end{align}
\end{enumerate}
\end{definition}
\begin{definition}
    A sequence of functionals $\{I_n\}_{n=1}^\infty$ is equicoercive if, for all $t\in \R$, there exists a compact subset $K_t\subseteq \H$ such that, for all $n$, $I_n^{-1}([-\infty,t])\subseteq K_t.$
\end{definition}
\begin{definition}
    Let $I,I_n:\H\to \overline{\R}$. We say that $I_n$ converges continuously to $I$ if, for every $u\in \H$ and every neighborhood $V$ of $I(u)$ in $\overline{\R}$, there exist $N\in \N$ and a neighborhood $U$ of $u$ such that $n\geq N$ and $u'\in U$ imply that $I_n(u')\in V.$
\end{definition}
Continuous convergence is a stronger notion of convergence than pointwise convergence and $\Gamma$-convergence, but is implied by uniform convergence of $I_n$ to $I$ in the case that $I$ is continuous \cite{dal2012introduction}. The following classical result can be found for instance in \cite{braides2006handbook}.
\begin{proposition}\label{thm:gamma convergence theorem}
    Let $I_n,I:\H \to \overline{\R}$ with $\Gamma$-$\lim_{n\to \infty}I_n=I$, and $\{I_n\}_{n=1}^\infty$ equicoercive. Then,
    \begin{align}
        \min_{\H} I=\lim_{n\to \infty}\inf_{\H}I_n,
    \end{align}
    and if $\{u_n\}_{n=1}^\infty$ is a precompact sequence such that $\lim_{n\to \infty} I_n(u_n)=\min_{\H}I,$ then every limit of a convergent subsequence of $\{u_n\}_{n=1}^\infty$ is a minimizer of $I.$ In particular, if each $I_n$ has a minimizer $u_n$, then any convergent subsequence of these minimizers is a minimizer of $I.$
\end{proposition}
Essentially, this proposition states that $\Gamma$-convergence is the correct notion of convergence of functionals to guarantee that their minimizers also converge. We now use our results from Subsection \ref{ssec:errorFEM} and the results in \cite{ayanbayev2021convergence} to prove convergence of the finite element discretized MAP estimators as the mesh is refined:

\begin{theorem}[Convergence of MAP Estimators]\label{thm:Gamma Convergence}
     Let the covariance operator $\Cm_0:\H\to \H$ 
 and its discrete approximation $C_0:\R_M^n\to\R_M^n$ satisfy Assumption \ref{assumption:errors} (i), and let the orthogonal projection operator $P$ satisfy Assumption \ref{assumption:errors} (iii). Assume additionally that the approximate forward model is such that $\Phi^{(n)}$ converges to $\Phi$ continuously as $n\to \infty$. Then, the corresponding sequence $\{I_{\emph{post}}^{(n)}\}_{n=1}^\infty$ of OM functionals given in \eqref{eq:discretized OM functional}   satisfies that  
\begin{align}
    \Gamma \text{-}\lim_{n\to \infty} I_{\emph{post}}^{(n)}=I_{\emph{post}},
\end{align}
and the cluster points as $n\to \infty$ of the MAP estimators of the posteriors $\mu_{\emph{post}}^n$ are MAP estimators of the continuum posterior $\mu_{\emph{post}}.$
\begin{proof}
    Assumptions \ref{assumption:errors} (i) and (iii) imply that, as $n \to \infty,$ $P^*C_0P$ converges to $\Cm_0$ in operator norm and $P^*Pm_0$ converges to $m_0$ in $\H$. 
    Therefore, we can apply Theorem 5.5 in \cite{ayanbayev2021convergence} to conclude that $I_0=\Gamma \text{-}\lim_{n\to \infty}I_0^{(n)}$ and that the sequence $\{I_0^{(n)}\}_{n=1}^\infty$ is equicoercive. Then, under the assumption that $\Phi^{(n)}\to \Phi$ continuously as $n\to \infty$, applying Theorem 6.1 in \cite{ayanbayev2021convergence} gives us that $\Gamma \text{-} \lim_{n\to \infty} I_{\text{post}}^{(n)}=I_{\text{post}},$ and the sequence $\{I^{(n)}_{\text{post}}\}_{n=1}^\infty$ is also equicoercive since $\Phi^{(n)}\geq 0$ in our setting. Thus, by Proposition \ref{thm:gamma convergence theorem} we conclude that 
     if $u_n$ is a minimizer of $I_{\text{post}}^{(n)},$ $ n\ge 1,$ then any convergent subsequence of $\{u_n\}_{n=1}^\infty$ converges to a minimizer of $I_{\text{post}}.$
    Since the minimizers of the OM functionals $I_{\text{post}}^{(n)}$ and  $I_{\text{post}}$ coincide with the MAP estimators of $\mu_{\text{post}}^{(n)}$ and $\mu_{\text{post}}$ respectively, we have hence shown     
    that any convergent subsequence of MAP estimators of the discretized posterior converges 
    to a MAP estimator of the continuum posterior.
\end{proof}
\end{theorem}
 We have demonstrated how  Assumptions \ref{assumption:errors} (i) and (iii) can be verified for finite element and graph-based discretizations. For nonlinear forward maps, verifying the continuous convergence of the discretized potential may not be straightforward. Lemma B.9 in  \cite{ayanbayev2021convergence} shows that if the approximate forward map is of the form $\F(P^*P\cdot),$ then continuous corresponding potentials converge continuously. The following example illustrates an important problem where the approximate forward model is not of this form; nonetheless, we will show in Theorem \ref{thm:Navier Stokes Gamma Convergence} that continuous convergence ---as well as Assumptions \ref{assumption:errors} (i) and (iii)--- can still be verified. 

\begin{example}[Eulerian Data Assimilation for the Navier-Stokes Equations] \label{ex:Eulerian Data Assimilation} 
    Eulerian data assimilation is concerned with learning the initial condition of a dynamical system from pointwise observations of the state at fixed spatial locations. This example briefly overviews Eulerian data assimilation for the Navier-Stokes equations, an important model problem in numerical weather forecasting. We refer to \cite{cotter2009bayesian,cotter2010approximation} for a more detailed discussion.


As described in \cite{robinson2001infinite,temam2012infinite}, the Navier-Stokes equations on the two-dimensional torus $\mathcal{D} = \mathbb{T}^2 = [0,1]^2$ can be written as an infinite-dimensional dynamical system
    \begin{align}\label{eq:navier stokes ODE}
        \frac{dv}{dt}=\nu \mathsf{A} v+ \mathsf{B}(v,v)= f, \quad v(0)=u,    \end{align}
on the Hilbert space
    \begin{align}
        \H=\left\{v\in L^2_{\text{per}}\left( \D\right) : \int_\mathcal{D} v \, dx=0, \nabla 
 \cdot v=0\right\}
    \end{align}
    equipped with the standard $L^2(\D)$ inner product.  The Stokes operator $\mathsf{A}$ in \eqref{eq:navier stokes ODE} is self-adjoint, positive, densely defined on $\H$, and has a complete set of eigenpairs, $\{(\lambda^{(i)},\psi^{(i)})\}_{i=1}^\infty,$ with increasingly sorted eigenvalues. The energy-conserving quadratic nonlinearity $\mathsf{B}(v,v)$ arises from projection under the Leray operator \cite{temam1995navier}. For $u\in \H$ and $f$ sufficiently regular \cite{robinson2001infinite,cotter2010approximation}, there exists a unique solution to \eqref{eq:navier stokes ODE} such that $v\in L^{\infty}\bigl(0,T;\H^{1+s}\bigr)\subset L^{\infty}\bigl(0,T;L^\infty(\D)\bigr).$
        
        We wish to determine the initial condition $u$ from noisy observations of the velocity field $v$ at time $t>0$ and fixed spatial locations $x_1,\ldots x_{K}\in \D,$ given by
\begin{align}
    y_k=v(x_k,t)+\eta_k, \quad k=1,\ldots, K,
\end{align}
where $\eta\sim \Nc(0,\Gamma).$ 
This Eulerian data assimilation task can be formulated as an inverse problem with nonlinear forward map
\begin{align}\label{eq:navier stokes forward map}
    \F(u) =\left(v(x_1,t)^T,\ldots,v(x_{K},t)^T \right)^T\in \R^{d_y},
\end{align}
where $d_y=2K.$ 
We consider a Gaussian prior measure $\mu_0\sim \Nc(m_0,\Cm_0)$, where $\Cm_0=\mathsf{A}^{-\alpha}$ with $\alpha>1$ and $m_0\in \H^{\alpha}(\D)$. Here, powers of $\mathsf{A}$ are defined spectrally. The posterior measure can then be characterized as in \eqref{eq:nonlinear posterior measure}. 

We discretize the inverse problem on the space spanned by the first $n$ eigenfunctions $\{\psi^{(i)}\}_{i=1}^n$ of the Stokes operator.
Note that since the eigenfunctions are orthonormal, the weighted inner product described in Subsection \ref{ssec:computationalframework} coincides with the usual Euclidean inner product. The discretized prior covariance operator $C_0:\R^n\to \R^n$ is then given by $C_0=\emph{diag}\left( (\lambda^{(1)})^{-\alpha},\ldots, (\lambda^{(n)})^{-\alpha}\right)$. The discretized forward map is given by a Galerkin approximation to the PDE solution, defined via the projection $P: \mathcal{H} \to \R^n$ given in \eqref{eq:discretizationmap}. We denote by $v^N$ the solution to the (finite-dimensional) ordinary differential equation
\begin{align}\label{eq:navier stokes galerkin}
    \frac{dv^N}{dt}+\nu \mathsf{A} v^N+P \mathsf{B}(v^N,v^N)=P f, \quad v^N(0)=Pu. 
\end{align}
We then define
\begin{align}
F(u)=\left(v^N(x_1,t),\ldots,v^N(x_{d_y},t) \right)^T.
\end{align}
Given these approximations, we have a discretized OM functional exactly of the form \eqref{eq:discretized OM functional}. 
\end{example}

For the Eulerian data assimilation problem summarized in Example \ref{ex:Eulerian Data Assimilation} and further detailed in \cite{cotter2010approximation},
we can verify the assumptions of Theorem \ref{thm:Gamma Convergence} to obtain the following result. The proof can be found in Appendix \ref{appendix:proofexample}. 
\begin{theorem}[MAP Estimation for Eulerian Data Assimilation]\label{thm:Navier Stokes Gamma Convergence}
    Let $f\in L^2(0,T;\H^s)$ with $s>0$ and consider the Eulerian data assimilation problem summarized in Example \ref{ex:Eulerian Data Assimilation}.
    Then, as $n\to \infty,$ the sequence $\{I_{\emph{post}}^{(n)}\}_{n=1}^\infty$ of discretized OM functionals satisfies that 
    \begin{align}
    \Gamma \text{-}\lim_{n\to \infty} I_{\emph{post}}^{(n)}=I_{\emph{post}},
\end{align}
and the cluster points of the MAP estimators of the posteriors $\mu_{\emph{post}}^n$ are MAP estimators of the continuum posterior $\mu_{\emph{post}}.$
\end{theorem}

\section{Conclusion and Future Directions}\label{sec:conclusions}
This paper analyzed a computational framework to solve infinite-dimensional Bayesian inverse problems. Working on a weighted inner product space, we have studied  finite element and graph-based discretizations in a unified framework, using Matérn-type priors and deconvolution forward models as guiding examples. We have established error guarantees for linear inverse problems, and analyzed ensemble Kalman algorithms and MAP estimators applicable in nonlinear inverse problems. In future work, our analysis may be extended to other types of discretizations, priors, and forward models. The generality of our presentation will facilitate the integration of numerical analysis for PDEs to obtain similar error guarantees for other priors and forward models. Additionally, our general approach to obtain error bounds for discretizations of covariance operators and forward maps will also facilitate the analysis under mesh refinement of other algorithms for nonlinear inverse problems.

\section*{Acknowledgments}
DSA is thankful for the support of NSF DMS-2027056, NSF DMS-2237628, DOE DE-SC0022232, and the BBVA Foundation. The authors are grateful to Yuming Chen for his generous feedback on a previous version of this manuscript.

\bibliographystyle{siamplain}
\bibliography{references}

\appendix

\section{Adjoints in Weighted Space}\label{sec:adjoint discussion}
Let $X$ and $Y$ be two Hilbert spaces. Recall that \cite{sunder1996functional}, given $A\in B(X,Y),$ the adjoint of $A$ is the unique map $A^*\in B(Y,X)$ such that
\begin{align} \label{eq:adjoint}
    \langle Ax,y\rangle_Y=\langle x,A^*y\rangle_X, \quad \quad \forall x \in X, \forall y \in Y.
\end{align}
In $\R^n$ equipped with the usual Euclidean inner product, the adjoint of a matrix coincides with the matrix transpose. However, for linear maps to and from the weighted inner product space $\R_M^n$, the structure of the adjoint specified by \eqref{eq:adjoint} no longer coincides with the matrix transpose. For the case where $A\in B(\R_M^n,\R_M^n)$, we have that $A^*$ must satisfy
\begin{align*}
    \vec{u}^TA^TM\vec{v}=\langle A\vec{u},\vec{v}\rangle_M=\langle \vec{u},A^*\vec{v}\rangle_M=\vec{u}^TMA^*\vec{v}, \quad \quad \forall \vec{u}, \vec{v} \in \R^n_M,
\end{align*}
from which it follows that $A^*=M^{-1}A^TM.$ Similarly, for $F\in B(\R_M^n,\R^{d_y})$, $F^{\natural}$ must satisfy
\begin{align*}
    \vec{u}^TF^Ty=\langle F\vec{u},y\rangle =\langle \vec{u},F^{\natural}y\rangle_M=\vec{u}^TMF^{\natural}y, \quad \quad \forall \vec{u}\in \R_M^n,\, \forall y\in \R^{d_y}, 
\end{align*}
from which it follows that $F^\natural=M^{-1}F^T.$

\section{Proofs Subsection \ref{ssec:errorFEM}}\label{AppendixFEMProofs}
This appendix contains the proofs of Theorems \ref{th:opNormPriorCov} and \ref{thm:forwardFEM} in Subsection \ref{ssec:errorFEM}. The proofs rely on classical finite element method convergence results from \cite{oden2012introduction}.
\begin{proof}[Proof of Theorem \ref{th:opNormPriorCov}]
We let $f\in L^2(\D)$ with $\|f\|_{L^2(\D)}\neq 0$. We proceed by induction on $\alpha$. For $\alpha=1$ we have that
            \begin{align*}
                \|\mathcal{A}^{-1}f-P^*A^{-1}Pf\|_{L^2(\D)}\leq ch^2\|\mathcal{A}^{-1}f\|_{H^2(\D)}\leq ch^2\|\mathcal{A}^{-1}\|_{op}\|f\|_{L^2(\D)},
            \end{align*}
            by \cite[Theorem 14.3.3]{brenner2008mathematical} and boundedness of the PDE solution operator from $L^2(\D)$ to $H^2(\D)$  (see \cite{evans2010partial} Chapter 6.3). Dividing both sides by $\|f\|_{L^2(\D)},$ we have that
            \begin{align*}
                \|\mathcal{A}^{-1}-P^*A^{-1}P\|_{op}\leq ch^2
            \end{align*}
            for a constant $c$ independent of $h$, completing the base case. Now, we assume that for some integer $\alpha\geq 1$, we have that
        \begin{align*}
            \|\mathcal{A}^{-\alpha}-P^*A^{-\alpha}P\|_{op}\leq ch^{2}.
        \end{align*}
        We then consider the quantity $\|\mathcal{A}^{-(\alpha+1)}f-P^*A^{-\alpha-1}Pf\|_{L^2(\D)}$. Then, by the triangle inequality we have that
        \begin{align*}
            \|\mathcal{A}^{-\alpha-1}f-P^*(K^{-1}M)^{\alpha+1}Pf\|_{L^2(\D)}&\leq \|\mathcal{A}^{-\alpha}(\mathcal{A}^{-1}-P^*A^{-1}P)f\|_{L^2(\D)} \\ 
            &+\|(\mathcal{A}^{-\alpha}-P^*A^{-\alpha}P)P^*A^{-1}Pf\|_{L^2(\D)}.
        \end{align*}
        The first of these terms can be controlled using the boundedness of $\mathcal{A}^{-1}$ as in the proof of the base case:
        \begin{align*}
            \|\mathcal{A}^{-\alpha}(\mathcal{A}^{-1}-P^*A^{-1}P)f\|_{L^2(\D)}\leq \|\mathcal{A}^{-1}\|_{op}^{\alpha}\|\mathcal{A}^{-1}f-P^*A^{-1}Pf\|_{L^2(\D)}\leq c\|\mathcal{A}^{-1}\|_{op}^{\alpha}h^2\|f\|_{L^2(\D)},
        \end{align*}
        and the second follows from submultiplicativity and the inductive hypothesis:
        \begin{align*}
            \|(\mathcal{A}^{-\alpha}-P^*A^{-\alpha}P)P^*A^{-1}Pf\|_{L^2(\D)}&\leq \|\mathcal{A}^{-\alpha}-P^*A^{-\alpha}P\|_{op}\|P^*A^{-1}Pf\|_{L^2(\D)} \\
            &\leq  c\|\mathcal{A}^{-1}\|_{op}h^2\|f\|_{L^2(\D)}.
        \end{align*}
        We conclude that
        \begin{align*}
            \|\mathcal{A}^{-\alpha-1}-P^*A^{-\alpha}P\|_{op}\lesssim h^2,
        \end{align*}
        completing the proof.
\end{proof}
\begin{proof}[Proof of Theorem \ref{thm:forwardFEM}]
    Let $u\in L^2(\D)$ with $\|u\|_{L^2(\D)}\neq 0$. We denote $v=\G u$ and $v_h= GPu.$ We then have that
    \begin{align*}
        \|\F u-FPu\|_2\leq \|\mathcal{O}\|_{op}\|v-v_h\|_{L^2(\D)}\leq \|\mathcal{O}\|_{op}\left( c_1h^{2}\|v\|_{H^{2}(\D)}+c_2\Delta t^2\|P u\|_{L^2(\D)}\right),
    \end{align*}
    where the first inequality uses the fact that $\mathcal{O}$ is a bounded linear operator from $L^2(\D)$ and the second uses Theorem \cite[Theorem 9.6]{oden2012introduction}, which is a classical finite element error result for the Crank-Nicolson method. Then, since $\G:L^2(\D)\to H^2(\D)$ is bounded and $\|P\|_{op}=1,$ we get that
    \begin{align*}
        \|\F u-FPu\|_2\leq c_1\|\mathcal{O}\|_{op}\|\G\|_{op}h^2\|u\|_{L^2(\D)}+c_2\|\mathcal{O}\|_{op}\Delta t^2\|u\|_{L^2(\D)}.
    \end{align*}
    Dividing both sides by $\|u\|_{L^2(\D)}$ gives the desired result.
\end{proof}
\section{Proofs Subsection \ref{ssec:errorgraph}}\label{AppendixGraphProofs}
This appendix contains the proofs of Theorems \ref{thm:opNormGraphCov} and \ref{thm:forwardgraph} in Subsection \ref{ssec:errorgraph}. The proofs rely on existing spectral convergence results from \cite{sanz2022spde}. We denote by $\{(\lambda^{(i)}_n, \psi^{(i)}_n)\}_{i=1}^n$ and  
 $\{(\lambda^{(i)},\psi^{(i)})\}_{i=1}^\infty$ the eigenpairs of the matrix $A$ defined in \eqref{eq:graphApprox} and the operator  $\mathcal{A}$ defined in  \eqref{eq:ellipticPDE},  with the eigenvalues in ascending order. We recall the spectral convergence results derived in \cite{sanz2022spde}:

\begin{proposition}\label{thm:graphEW}
    Suppose $k=k_n$ is such that $h_n\sqrt{\lambda^{(k)}}\ll 1$ for large $n$. Then,
    \begin{align}
        \frac{|\lambda_n^{(k)}-\lambda^{(k)}|}{\lambda^{(k)}}\leq c\left(\frac{\eps_n}{h_n}+h_n\sqrt{\lambda^{(k)}} \right),
    \end{align}
    where $c$ is a constant depending on $\M$, $\Theta$, and $b.$
\end{proposition}

\begin{proposition}\label{thm:graphEV}
    Let $\lambda$ be an eigenvalue of $\mathcal{A}$ with multiplicity $\ell$. Suppose that $h_n\sqrt{\lambda^{(k_n)}}\ll 1$ and $\eps_n\ll h_n$ for $n$ large. Let $\psi_n^{(k_n)},\ldots,\psi_n^{(k_n+\ell-1)}$ be orthonormal eigenvectors of $A$ associated with eigenvalues $\lambda_n^{(k_n)},\ldots,\lambda_n^{(k_n+\ell-1)}$. Then, there exist orthonormal eigenfunctions $\psi^{(k_n)},\ldots,\psi^{(k_n+\ell-1)}$ of $\mathcal{A}$ so that, for $j=k_n,\ldots,k_n+\ell-1,$ 
    \begin{align}
        \|P^*\psi_n^{(j)}-\psi^{(j)}\|_{L^2(\M)}\leq  c j^{3/2}\left(\frac{\eps_n}{h_n}+h_n\sqrt{\lambda^{(j)}} \right)^{1/2},
    \end{align}
    where $c$ is a constant depending on $\M$, $\Theta$, and $b.$
\end{proposition}

\begin{proof}[Proof of Theorem \ref{thm:opNormGraphCov}]
        Let $f\in L^2(\M)$. Note that $k_n$ was chosen such that we can apply \eqref{thm:graphEV} and \eqref{thm:graphEW} to quantify the errors in the eigenvalues and eigenvectors up to $i=k_n$. We denote
        \begin{align*}
             u&=\Cm_0f=\mathcal{A}^{-\alpha}f=\sum_{i=1}^\infty(\lambda^{(i)})^{-\alpha}\langle f, \psi^{(i)}\rangle_{L^2(\M)}\psi^{(i)}, \\
             u_n&=P^*(\Delta_n^{\Theta}+B_n)^{-\alpha}Pf=\sum_{i=1}^n(\lambda^{(i)}_n)^{-\alpha}\langle Pf,\psi^{(i)}_n\rangle_M P^*\psi_n^{(i)}.
        \end{align*}
        We remark that $\langle Pf,\psi^{(i)}_n\rangle_M=\langle f,P^*\psi^{(i)}_n\rangle_{L^2(\M)},$ so we can also write
        $$
        u_n=\sum_{i=1}^n(\lambda^{(i)}_n)^{-\alpha}\langle f,P^*\psi^{(i)}_n\rangle_{L^2(\M)} P^*\psi_n^{(i)}.
        $$
         We want to bound $\|u-u_n\|_{L^2(\M)}$. To do so, introduce four intermediate quantities:
         \begin{align*}
             u_n^{k_n} &= \sum_{i=1}^{k_n}(\lambda_n^{(i)})^{-\alpha}\langle f,P^*\psi^{(i)}_n\rangle_{L^2(\M)} P^*\psi_n^{(i)}, \\
             \tilde{u}^{k_n}_n &=\sum_{i=1}^{k_n}(\lambda^{(i)})^{-\alpha}\langle f,P^*\psi^{(i)}_n\rangle_{L^2(\M)} P^* \psi_n^{(i)}, \\
             \tilde{u}^{k_n}&=\sum_{i=1}^{k_n}(\lambda^{(i)})^{-\alpha}\langle f,P^*\psi^{(i)}_n\rangle_{L^2(\M)} \psi^{(i)}, \\
             u^{k_n}&=\sum_{i=1}^{k_n}(\lambda^{(i)})^{-\alpha}\langle f,\psi^{(i)}\rangle_{L^2(\M)}\psi^{(i)}.
         \end{align*}
        We will bound the difference between each two consecutive functions. By Weyl's law \cite[Theorem 72]{canzani2013analysis}, we have that
        \begin{align}\label{eq:bound1}
        \|u-u^{k_n}\|_{L^2(\M)}\leq \left(\sum_{i=k_n+1}^\infty (\lambda^{(i)})^{-2\alpha} \right)^{\frac{1}{2}}\lesssim \left(\sum_{i=k_n+1}^\infty i^{-\frac{4\alpha}{d}} \right)^{\frac{1}{2}}\lesssim \left( \int_{k_n}^\infty x^{-\frac{4\alpha}{d}}\right)^{\frac{1}{2}}\lesssim k_n^{\frac{1}{2}-\frac{2\alpha}{d}}.
        \end{align}
        By \eqref{thm:graphEW} and Weyl's law, we have that $\lambda_n^{(k_n)} \gtrsim \lambda^{(k_n)}\gtrsim k_n^{\frac{2}{d}}$, from which we get 
        \begin{align}\label{eq:bound2}
        \|u_n-u_n^{k_n}\|_{L^2(\M)}=\left(\sum_{i=k_n+1}^n(\lambda_n^{(i)})^{-2\alpha}|\langle f,P^*\psi^{(i)}_n\rangle_{L^2(\M)}|^2\right)^{\frac12}\leq (\lambda_n^{(k_n)})^{-\alpha}\lesssim k_n^{-\frac{2\alpha}{d}}.
        \end{align}
        Next, we note that since $\lambda_n^{(i)}$ and $\lambda^{(i)}$ are bounded from below by $\min_{x\in \M} b(x)>0$ and $x^{-\alpha}$ is continuously differentiable away from zero, the mean value theorem guarantees that 
        $$
        \left|(\lambda_n^{(i)})^{-\alpha}-(\lambda^{(i)})^{-\alpha} \right|\leq \alpha |\xi|^{-\alpha-1}\left|\lambda_{n}^{(i)}-\lambda^{(i)} \right|
        $$
        for some $\xi$ between $\lambda_n^{(i)}$ and $\lambda^{(i)}$. This inequality combined with Proposition \ref{thm:graphEW} implies that
        $$
        \left|(\lambda_n^{(i)})^{-\alpha}-(\lambda^{(i)})^{-\alpha} \right|\leq \alpha \left(\lambda_n^{(i)} \wedge \lambda^{(i)} \right)^{-\alpha-1}\left|\lambda_{n}^{(i)}-\lambda^{(i)} \right|\lesssim (\lambda^{(i)})^{-\alpha}\left(\frac{\eps_n}{h_n}+h_n \sqrt{\lambda^{(i)}} \right),
        $$
        for $i=1,\ldots,k_n.$ Using this fact, we get that
        \begin{align}
        \begin{split}\label{eq:bound3}
        \|u_n^{k_n}-\tilde{u}_n^{k_n}\|_{L^2(\M)} &\leq \left(\sum_{i=1}^{k_n}\left[(\lambda_n^{(i)})^{-\alpha}-(\lambda^{(i)})^{-\alpha} \right]^2 \right)^{\frac{1}{2}}  \lesssim \left(\sum_{i=1}^{k_n}(\lambda^{(i)})^{-2\alpha} \left(\frac{\eps_n}{h_n}+h_n\sqrt{\lambda^{(i)}} \right)^2\right)^{\frac{1}{2}} \\
        &\lesssim \left(\frac{\eps_n}{h_n}+h_n \right)\left(\sum_{i=1}^{k_n} (\lambda^{(i)})^{-2\alpha+1} \right)^{\frac{1}{2}}\lesssim \left( \frac{\eps_n}{h_n}+h_n \right),
        \end{split}
        \end{align}
        where the last step follows if $\alpha>\frac{d}{4}+\frac{1}{2}$. This is implied by the requirement that $\alpha>(5d + 1)/4$. Next, we use \eqref{thm:graphEV} to deduce that
        \begin{align}
        \begin{split}\label{eq:bound4}
        \|\tilde{u}^{k_n}-\tilde{u}^{k_n}_n\|_{L^2(\M)} &\lesssim \sum_{i=1}^{k_n}(\lambda^{(i)})^{-\alpha}\|\psi^{(i)}-P^*\psi^{(i)}_n\|_{L^2(\M)} \lesssim \sum_{i=1}^{k_n}(\lambda^{(i)})^{-\alpha}i^{\frac{3}{2}}\sqrt{\frac{\eps_n}{h_n}+h_n\sqrt{\lambda^{(i)}}} \\
        &\lesssim \sqrt{\frac{\eps_n}{h_n}+h_n}\sum_{i=1}^{k_n}i^{\frac{3}{2}}(\lambda^{(i)})^{-\alpha+\frac{1}{4}}\lesssim \sqrt{\frac{\eps_n}{h_n}+h_n}\sum_{i=1}^{k_n} i^{\frac{3}{2}-\frac{2\alpha}{d}+\frac{1}{2d}}\lesssim \sqrt{\frac{\eps_n}{h_n}+h_n},
        \end{split}
        \end{align}
        where the last step follows if $\alpha>\frac{5}{4}d+\frac{1}{4}.$ Finally, we use Cauchy-Schwartz to get that
        \begin{align}
        \begin{split}\label{eq:bound5}
        \|u^{k_n}-\tilde{u}^{k_n}\|_{L^2(\M)} &= \left\|\sum_{i=1}^{k_n}\langle f,\psi^{(i)}-P^*\psi^{(i)}_n\rangle_{L^2(\M)}(\lambda^{(i)})^{-\alpha}\psi^{(i)}\right\|_{L^2(\M)}\\
        &\lesssim \sum_{i=1}^{k_n}(\lambda^{(i)})^{-\alpha}\|\psi^{(i)}-P^*\psi^{(i)}_n\|_{L^2(\M)}\lesssim  \sqrt{\frac{\eps_n}{h_n}+h_n} \,\, ,
        \end{split}
        \end{align}
        where the last step follows by the same argument as for the previous expression if $\alpha>\frac{5}{4}d+\frac{1}{4}.$ Combining \eqref{eq:bound1},  \eqref{eq:bound2},  \eqref{eq:bound3},  \eqref{eq:bound4}, and \eqref{eq:bound5}, and taking $k_n \asymp n^{\frac{2d+1}{4\alpha}}$ we see that the error is dominated by
        $$
        \|\Cm_0-P^*(\Delta_n^{\Theta}+B_n)^{-\alpha}P\|_{op}\lesssim k_n^{\frac{1}{2}-\frac{2\alpha}{d}}+\sqrt{\frac{\eps_n}{h_n}+h_n}\lesssim (\log n)^{\frac{c_d}{4}}n^{-\frac{1}{4d}},
        $$    
        as desired.
    \end{proof}

 \begin{proof}[Proof of Theorem \ref{thm:forwardgraph}]
        Let $u\in L^2(\M)$ with $\|u\|_{L^2(\M)}\neq 0.$ We write
        \begin{align}
          v&=\G u=\sum_{i=1}^\infty \exp \left(-\lambda^{(i)} \right)\langle u,\psi^{(i)}\rangle_{L^2(\M)}\psi^{(i)}, \\ 
           v_n&=G_n(u_n) = \sum_{i=1}^n\exp \left(-\lambda^{(i)}_n \right)\langle P u,\psi^{(i)}_n\rangle_{L^2(\M)}P^*\psi^{(i)}_n.
        \end{align}
        We then consider the quantity
        \begin{align}\label{eq:graph forward model terms}
            \|\F u- FPu\|_2=\|\mathcal{O}v-\mathcal{O}^nv_n\|_2\leq \|\mathcal{O}\|_{op}\|v-v_n\|_{L^2(\D)}+\|\mathcal{O}v_n\mathcal{O}^nv_n\|_{L^2(\M)}.
        \end{align}
        We first want to bound $\|v-v_n\|_{L^2(\M)}.$ To do so, we introduce the following intermediate quantities:
        \begin{align*}
             v_n^{k_n} &= \sum_{i=1}^{k_n}\exp \left( -\lambda^{(i)}_n\right)\langle u,P^*\psi^{(i)}_n\rangle_{L^2(\M)} P^*\psi_n^{(i)}, \\
             \tilde{v}^{k_n}_n &=\sum_{i=1}^{k_n}\exp\left(-\lambda^{(i)}\right)\langle u,P^*\psi^{(i)}_n\rangle_{L^2(\M)} P^* \psi_n^{(i)}, \\
             \tilde{v}^{k_n}&=\sum_{i=1}^{k_n}\exp \left(-\lambda^{(i)}\right)\langle u,P^*\psi^{(i)}_n\rangle_{L^2(\M)} \psi^{(i)}, \\
             v^{k_n}&=\sum_{i=1}^{k_n}\exp \left(-\lambda^{(i)}\right)\langle u,\psi^{(i)}\rangle_{L^2(\M)}\psi^{(i)}.
         \end{align*} 
         We proceed to bound the difference between each two consecutive functions, as in the proof of Theorem \ref{thm:opNormGraphCov}. By Weyl's law, we have that
         \begin{align*}
             \|v-v^{k_n}\|_{L^2(\M)}\leq \left(\sum_{i=k_n+1}^{\infty}\exp \left( -\lambda^{(i)}\right) \right)^{\frac{1}{2}}  \hspace{-0.2cm}\lesssim \left(\sum_{i=k_n+1}^\infty \exp \left(-i^\frac{2}{d} \right)\right)^{\frac{1}{2}} \hspace{-0.2cm} \lesssim \left(\int_{k_n}^\infty \exp\left(-x^{\frac{2}{d}} \right)  dx\right)^{\frac12}.
         \end{align*}
         This integral does not have a nice closed form solution. However, for any fixed $\alpha>0$, it holds that $\exp(-x^{2/d})<x^{-4\alpha/d}$ for $x$ large enough. Assuming that $n$ is large enough such that this inequality holds for $x=k_n$, we then have the bound
         \begin{align}\label{eq:heat bound 1}
             \|v-v^{k_n}\|_{L^2(\M)}\leq \left(\int_{k_n}^\infty x^{-4\alpha/d} \, dx\right)^{\frac12}\lesssim k_n^{\frac{1}{2}-\frac{2\alpha}{d}}.
         \end{align}
         While this bound is not necessarily sharp, the operator-norm error will still be dominated by the error in approximating the eigenfunctions, so the final bound will still be the same. By \eqref{thm:graphEW} and Weyl's law, we have that $\lambda_n^{(k_n)} \gtrsim \lambda^{(k_n)}\gtrsim k_n^{\frac{2}{d}}$, from which we get 
        \begin{align}\label{eq:heat bound 2}
        \begin{split}
        \|v_n-v_n^{k_n}\|_{L^2(\M)}&=\left(\sum_{i=k_n+1}^n\exp\left(-\lambda_n^{(i)}\right)|\langle u,P^*\psi^{(i)}_n\rangle_{L^2(\M)}|^2\right)^{\frac12}  \\
        &\leq \exp\left(-\frac{1}{2}\lambda_n^{(k_n)}\right)  \lesssim \exp\left(-\frac{1}{2}k_n^{\frac{2}{d}}\right).
        \end{split}
        \end{align}
         Next, we note that since $e^{-x}$ is continuously differentiable we have that, for each $i=1,\ldots,k_n,$ 
         \begin{align*}
             \left| \exp \bigl(-\lambda_n^{(i)} \bigr) - \exp \bigl( -\lambda^{(i)}\bigr)\right|\lesssim \exp \Bigl( -(\lambda_n^{(i)} \wedge \lambda^{(i)}) \Bigr)\left|\lambda_n^{(i)}-\lambda^{(i)} \right|\lesssim \exp\bigl(-\lambda^{(i)} \bigr)\left(\frac{\eps_n}{h_n}+h_n\sqrt{\lambda^{(i)}} \right).
         \end{align*}
          From this, we get that
         \begin{align}\label{eq:heat bound 3}
         \begin{split}
             \|v_n^{k_n}-\tilde{v}_n^{k_n}\|_{L^2(\M)}& \leq \left(\sum_{i=1}^{k_n}\left|\exp \left(-\lambda_n^{(i)} \right) - \exp \left( -\lambda^{(i)}\right) \right| \right)^{\frac12} \\
             &\lesssim \left(\exp\left(-2\lambda^{(i)} \right)\left(\frac{\eps_n}{h_n}+h_n\sqrt{\lambda^{(i)}} \right)^2 \right)^{\frac12} \\
            & \lesssim \left(\frac{\eps_n}{h_n}+h_n \right)\left(\sum_{i=1}^{k_n}\frac{\lambda^{(i)}}{\exp\left(2\lambda^{(i)} \right)}\right)^{\frac12}\lesssim \left(\frac{\eps_n}{h_n}+h_n \right).
            \end{split}
         \end{align}
          Next, we use \eqref{thm:graphEV} to deduce that
        \begin{align}
        \begin{split}\label{eq:heat bound 4}
        \|\tilde{v}^{k_n}-\tilde{v}^{k_n}_n\|_{L^2(\M)} &\lesssim \sum_{i=1}^{k_n}\exp\left(-\lambda^{(i)}\right)\|\psi^{(i)}-P^*\psi^{(i)}_n\|_{L^2(\M)} \\
        &\lesssim \sum_{i=1}^{k_n}\exp\left(-\lambda^{(i)}\right)i^{\frac{3}{2}}\sqrt{\frac{\eps_n}{h_n}+h_n\sqrt{\lambda^{(i)}}} \\
        &\lesssim \sqrt{\frac{\eps_n}{h_n}+h_n}\sum_{i=1}^{k_n}i^{\frac{3}{2}}\exp\left(-\lambda^{(i)}\right)(\lambda^{(i)})^{\frac{1}{4}}\\
        &\lesssim \sqrt{\frac{\eps_n}{h_n}+h_n}\sum_{i=1}^{k_n} \frac{i^{\frac{3}{2}+\frac{1}{2d}}}{\exp\left( i^{2/d}\right)}\lesssim \sqrt{\frac{\eps_n}{h_n}+h_n}.
        \end{split}
        \end{align}
        Finally, we use Cauchy-Schwartz to get that
        \begin{align}
        \begin{split}\label{eq:heat bound 5}
        \|v^{k_n}-\tilde{v}^{k_n}\|_{L^2(\M)} &= \left\|\sum_{i=1}^{k_n}\langle u,\psi^{(i)}-P^*\psi^{(i)}_n\rangle_{L^2(\M)}\exp\left(-\lambda^{(i)}\right)\psi^{(i)}\right\|_{L^2(\M)}\\
        &\lesssim \sum_{i=1}^{k_n}\exp\left(-\lambda^{(i)}\right)\|\psi^{(i)}-P^*\psi^{(i)}_n\|_{L^2(\M)}\lesssim  \sqrt{\frac{\eps_n}{h_n}+h_n} \,\, ,
        \end{split}
        \end{align}
        where the last step follows by the same argument as for \eqref{eq:heat bound 4}. Combining \eqref{eq:heat bound 1}, \eqref{eq:heat bound 2}, \eqref{eq:heat bound 3}, \eqref{eq:heat bound 4}, and \eqref{eq:heat bound 5}, and taking $k_n \asymp n^{\frac{2d+1}{4\alpha}}$, we see that the error is dominated by
        \begin{align}\label{eq:heat equation graphs bound}
             \|v-v_n\|_{op}\lesssim k_n^{\frac{1}{2}-\frac{2\alpha}{d}}+\sqrt{\frac{\eps_n}{h_n}+h_n} \lesssim (\log n)^{\frac{c_d}{4}}n^{-\frac{1}{4d}},
        \end{align}
        as desired. Taking the supremum of both sides over all $\|u\|_{L^2(\M)}=1$, we get that
        \begin{align}\label{eq:heat op bound graphs}
             \|\G-P^*G_nP\|_{op}\lesssim k_n^{\frac{1}{2}-\frac{2\alpha}{d}}+\sqrt{\frac{\eps_n}{h_n}+h_n} \lesssim (\log n)^{\frac{c_d}{4}}n^{-\frac{1}{4d}}.
        \end{align}
        Now we proceed to bound $\|\mathcal{O}v_n-\mathcal{O}^nv_n\|_2.$ Recall that, for $k=1,\ldots d_y$, we have
        \begin{align*}
            \left[\mathcal{O}v_n\right]_k=\int_{B_{\delta}(x_k)\cap \M}\sum_{i=1}^nv_n(x)\mathbb{1}_{U_i}(x)d\gamma=\langle v_n,\mathbb{1}_{B_\delta(x_k)\cap \M}(x)\rangle_{L^2(\M)}
        \end{align*}
        and
        \begin{align*}
            \left[\mathcal{O}^nv_n\right]_k=\frac{1}{n}\sum_{i=1}^nv_n(x_i)\mathbb{1}_{B_{\delta}(x_k)}(x_i)=\langle P v_n, \mathbb{1}_{B_\delta(x_k)\cap \M_n}\rangle_M, 
        \end{align*}
        where $\mathbb{1}_{B_\delta(x_k)\cap \M_n}\in \R_M^n$ is the vector with $i^{th}$ entry $1$ if $x_i\in B_{\delta}(x_k)$ for all $x_i\in M_n.$ Note that here the points $x_i$ are in the point cloud $\M_n$ whereas the $x_k$ are the points at which the balls of radius $\delta$ in the observation operator are centered. By our definition of $P$, it follows that we can write
        \begin{align*}
            \left[\mathcal{O}^nv_n \right]_k=\langle v_n,P^*\mathbb{1}_{B_\delta (x_k)\cap \M_n}(x)\rangle_{L^2(\M)}.
        \end{align*}
        Consequently, we can apply Cauchy-Schwarz to get the bound
        \begin{align}\label{eq:observation term 1}
        \begin{split}
            \|\mathcal{O}v_n-\mathcal{O}^nv_n\|_2^2 &=\sum_{k=1}^{d_y}\langle v_n, \mathbb{1}_{B_\delta(x_k)\cap \M}(x)-P^*\mathbb{1}_{B_{\delta}(x_k)\cap \M_n}(x)\rangle_{L^2(\M)}\\ 
            &\leq \sum_{k=1}^{d_y}\|v_n\|_{L^2(\M)}^2\|\mathbb{1}_{B_{\delta}(x_k)\cap \M}-P^*\mathbb{1}_{B_\delta(x_k)\cap \M_n}\|_{L^2(\M)}^2.
            \end{split}
        \end{align}
        From \eqref{eq:heat equation graphs bound}, we have that 
        \begin{align*}
        \|v_n\|_{L^2(\M)}\leq \|v\|_{L^2(\M)}+c(\log n)^{\frac{c_d}{4}}n^{-\frac{1}{4d}}\leq \|\mathcal{G}\|_{op}\|u\|_{L^2(\M)}+c(\log n)^{\frac{c_d}{4}}n^{-\frac{1}{4d}}.
        \end{align*}
        Hence, for $n$ large enough it holds that
        \begin{align}\label{eq:observation term 2}
           \|v_n\|_{L^2(\M)} \leq 2\|\G\|_{op}\|u\|_{L^2(\M)}.
        \end{align} 
        It then remains to bound $\|\mathbb{1}_{B_{\delta}(x_k)\cap \M}-P^*\mathbb{1}_{B_\delta(x_k)\cap \M_n}\|_{L^2(\M)}^2$. To do so, we remark that we can write
        $$
        P^*\mathbb{1}_{B_{\delta}(x_k)\cap \M_n}(x)=\sum_{i=1}^n\mathbb{1}_{U_i}(x)\mathbb{1}_{B_{\delta}(x_k)}(x_i)=\sum_{i=1}^n\mathbb{1}_{U_i}(x)\mathbb{1}_{B_\delta(x_k)}\left(T_n(x)\right),
        $$
        since $T_n(x)=x_i$ for all $x\in U_i$. Since the sets $U_i$ partition $\M$, we have that
        $$
        P^*\mathbb{1}_{B_{\delta}(x_k)\cap \M_n}(x)=\mathbb{1}_{B_{\delta}(x_k)\cap \M}\left(T_n(x)\right).
        $$
        We then see that $\|\mathbb{1}_{B_{\delta}(x_k)\cap \M}-P^*\mathbb{1}_{B_\delta(x_k)\cap \M_n}\|_{L^2(\M)}^2$ is exactly the measure of the set where the indicator functions $\mathbb{1}_{B_{\delta}(x_k)\cap \M}\left(T_n(x)\right)$ and $\mathbb{1}_{B_{\delta}(x_k)\cap \M}\left(x\right)$ are not equal. Thus, we simply need to bound the measure of the sets 
        \begin{align*}E_k= \Bigl\{x\in \M: x\in B_{\delta}(x_k), T_n(x)\notin B_\delta(x_k)\Bigr\}\cup
        \Bigl\{x\in \M: x\notin B_{\delta}(x_k), T_n(x)\in B_\delta(x_k) \Bigr\}.
        \end{align*}
        We remark that for any $x\in U_i\subset B_{\delta}(x_k)$, we also then have that $T_n(x)\in U_i\subset  B_{\delta}(x_k)$. For $U_i\subset B_\delta(x_k)$, we have $U_i\subset E_k^C.$ Similarly, for any $x\in U_i\subset B_\delta(x_k)^C$, we also have that $T_n(x)\in U_i$, so $x\notin B_\delta(x_k)$ and $T_n(x)\notin B_\delta(x_k)$. Hence, for $U_i\subset B_\delta(x_k)^C$, we have $U_i \subset E_k^C.$ Since any $U_i$ that are entirely within $B_\delta(x_k)$ and any $U_i$ that are entirely outside of $B_\delta(x_k)$ are contained in the complement of $E_k$ and the sets $U_i$ partition $\M$, we must have that $E_k$ is contained in the union of the sets $U_i$ that have nontrivial intersection with both $B_\delta(x_k)$ and $B_\delta(x_k)^C.$ That is, $E_k$ is contained in the union of the sets $U_i$ that intersect the boundary of $B_\delta(x_k)\cap \M$:
        \begin{align*}
            E_k\subset \tilde{E}_k=\bigcup_{i:U_i\cap \partial(B_\delta(x_k)\cap \M) }U_i.
        \end{align*}
        From Theorem \ref{thm:transportMaps}, we have that $\sup_{x\in \M}d_{\M}\left(x,T_n(x)\right)=\eps_n$. That is, the geodesic distance between any two points in a set $U_i$ is at most $2\eps_n$. This fact, combined with the assumption that $\text{length}\left(\partial(B_\delta(x_k)\cap \M)\right)=C_{\delta,k},$ gives us that 
        \begin{align}\label{observation term 3}
            \gamma \left(\tilde{E}_k \right)\leq 2\eps_n \text{length}\left(\partial(B_\delta(x_k)\cap \M)\right)= 2C_{\delta,k}\eps_n.
        \end{align}
        The constant $C_{\delta,k}$ depends on $\delta$, $\M$, and $k$, but not $n$. This inequality, combined with \eqref{eq:observation term 1} and \eqref{eq:observation term 2} gives us that
        \begin{align}\label{eq:observation bound}
            \|\mathcal{O}v_n-\mathcal{O}^nv_n\|_2^2\leq 2d_y\|\G\|_{op}C_{\delta,d}\eps_n\lesssim d_y\frac{(\log n)^{c_d}}{n^{1/d}}.
        \end{align}
        Plugging \eqref{eq:heat equation graphs bound} and \eqref{eq:observation bound} into \eqref{eq:graph forward model terms} and taking the supremum of both sides over all $\|u\|_{L^2(\M)}=1$ yields
        \begin{align}
            \|\F-FP\|_{op}\lesssim (\log n)^{\frac{c_d}{4}}n^{-\frac{1}{4d}}+\sqrt{d_y}(\log n)^{\frac{c_d}{2}}n^{-\frac{1}{2d}}.
        \end{align}
        In our setting where the dimension of the observations is fixed, the error will be dominated by the first of these two terms.
    \end{proof}
    
\section{Proof of Theorem \ref{thm:Navier Stokes Gamma Convergence}}\label{appendix:proofexample}
\begin{proof}[Proof of Theorem \ref{thm:Navier Stokes Gamma Convergence}]
We want to verify the assumptions of Theorem \ref{thm:Gamma Convergence} for the setting of Eulerian data assimilation described in Example \ref{ex:Eulerian Data Assimilation}. First, we verify that the discretized prior mean converges to the continuum prior mean in $\H$. Since $m_0\in \H^{\alpha}(\D)$, we have from Equation (4.7) in \cite{cotter2010approximation} that
    \begin{align*}\label{eq:NS mean convergence}
        \|m_0-P^*Pm_0\|_{L^2(\D)}^2\leq \frac{1}{n^{2\alpha}}\|m_0\|_{\H^\alpha(\D)}^2,
    \end{align*}
    which verifies convergence of the discretized prior mean. Next, we verify operator norm convergence of the discretized prior covariance. For any $u\in \H$ with $\|u\|_{L^2(\D)}=1,$ 
\begin{align*}
    \|\Cm_0u-P^*C_0Pu\|_{L^2(\D)}^2=\sum_{i=n+1}^\infty (\lambda^{(i)})^{-2\alpha}\langle u, \psi^{(i)}\rangle_{L^2(\D)}^2\leq \sum_{i=n+1}^\infty \left( \lambda^{(i)}\right)^{-2\alpha} \hspace{-0.1cm} \lesssim \sum_{i=n+1}^\infty i^{-2\alpha}\lesssim \frac{1}{n^{2\alpha-1}}.
\end{align*}
Hence, we have that
\begin{align}\label{eq:NS covariance convergence}
    \|\Cm_0-P^*C_0P\|_{op}\lesssim \frac{1}{n^{\alpha-\frac{1}{2}}},
\end{align}
and since $\alpha>1$ we have shown operator norm convergence of the discretized prior covariance.

It remains to show that the discretized potential $\Phi^{(n)}$ converges continuously to $\Phi$ as $n\to \infty$. That is, we want to show that for every $u\in \H$ and every $\eps>0$, there exist $N\in \mathbb{N}$ and $\delta>0$ such that if $n\geq N$ and $u'\in B_{\delta}(u)$, then $\left|\Phi^{(n)}(u')-\Phi(u) \right|<\eps.$ Fix $u\in \H$ and $\eps>0$. From Lemma 3.2 in \cite{cotter2009bayesian}, we have that for any $u,u'\in \H$ and $f\in L^2(0,T;\H)$, there exists a constant $L(\|u\|_{L^2(\D)},\|u'\|_{L^2(\D)},|f|_0,t_0,\Gamma)$ such that 
\begin{align*}
    \|\F(u)-\F(u')\|_{\Gamma^{-1}}\leq L\|u-u'\|_{L^2(\D)}
\end{align*}
provided that $t>t_0>0.$ We then see that
\begin{align*}
    \Phi(u)-\Phi(u') &=\frac{1}{2}\left( \|y-\F(u)\|_{\Gamma^{-1}}-\|y-\F(u')\|_{\Gamma^{-1}}\right)\left( \|y-\F(u)\|_{\Gamma^{-1}}+\|y-\F(u')\|_{\Gamma^{-1}}\right)\\
        &\leq \frac{1}{2}\|\F(u)-\F(u')\|_{\Gamma^{-1}}\left( \|y-\F(u)\|_{\Gamma^{-1}}+\|y-\F(u')\|_{\Gamma^{-1}}\right)\\
        &\leq\frac{1}{2} \|\F(u)-\F(u')\|_{\Gamma^{-1}}\left(2\|y\|_{\Gamma^{-1}}+\|\F(u)\|_{\Gamma^{-1}}+\|\F(u')\|_{\Gamma^{-1}}\right) \\
        &\leq \frac{1}{2}L\|u-u'\|_{L^2(\D)}\left(2\|y\|_{\Gamma^{-1}}+2\|\F(u)\|_{\Gamma^{-1}}+L\|u-u'\|_{L^2(\D)}. \right).
    \end{align*}
    The exact same argument can be carried out for $\Phi(u')-\Phi(u)$ and yields the same upper bound, from which we conclude that
    \begin{align}\label{eq:phi bound 1}
        \left| \Phi(u)-\Phi(u')\right|\leq \frac{1}{2}L\|u-u'\|_{L^2(\D)}\left(2\|y\|_{\Gamma^{-1}}+2\|\F(u)\|_{\Gamma^{-1}}+Lt_0\|u-u'\|_{L^2(\D)}\right).
    \end{align}
    From Lemma 4.2 in \cite{cotter2010approximation}, we have that for the Galerkin approximation $v^N$ in \eqref{eq:navier stokes galerkin} to the true solution $v$ in \eqref{eq:navier stokes ODE}, there exists a constant $\tilde{L}\left(\|u\|_{L^2(\D)},t_0,\Gamma \right)$ such that, for any $t>t_0>0,$
\begin{align*}
    \|v(t)-v^N(t)\|_{L^\infty(\D)}\leq \tilde{L}(\|u\|_{L^2(\D)})\psi(n),
\end{align*}
where $\psi(n)\to 0$ as $n\to \infty$ and $\tilde{L}(\|u\|_{L^2(\D)})$ is continuous and increasing in $\|u\|_{L^2(\D)}$. It follows from this inequality that
\begin{align*}
    \|\F(u')-F(u')\|_{\Gamma^{-1}}\leq \tilde{L}(\|u'\|_{L^2(\D)})\psi(n)\leq \tilde{L}(\|u\|_{L^2(\D)}+\delta)\psi(n).
\end{align*}
Then, we observe that
\begin{align*}
    \Phi(u')-\Phi^{(n)}(u') &=\frac{1}{2}\left( \|y-\F(u')\|_{\Gamma^{-1}}-\|y-F(u')\|_{\Gamma^{-1}}\right)\left( \|y-\F(u')\|_{\Gamma^{-1}}+\|y-F(u')\|_{\Gamma^{-1}}\right)\\
    &\leq \frac{1}{2}\|\F(u')-F(u')\|_{\Gamma^{-1}}\left( \|y-\F(u')\|_{\Gamma^{-1}}+\|y-F(u')\|_{\Gamma^{-1}}\right)\\
    &\leq \frac{1}{2}\|\F(u')-F(u')\|_{\Gamma^{-1}}\left(2\|y\|_{\Gamma^{-1}}+\|\F(u')\|_{\Gamma^{-1}}+\|F(u')\|_{\Gamma^{-1}}\right)\\
    &\leq \frac{1}{2}\tilde{L}(\|u\|_{L^2(\D)}+\delta)\psi(n)\left(2\|y\|_{\Gamma^{-1}}+2\|\F(u)\|_{\Gamma^{-1}}+\tilde{L}(\|u\|_{L^2(\D)}+\delta)\psi(n) \right).
\end{align*}
Again, the same argument carried out for $\Phi^{(n)}(u')-\Phi(u')$ yields the same upper bound, and we conclude that  
\begin{align}\label{eq:phi bound 2}
    \left|\Phi(u')-\Phi^{(n)}(u') \right|\leq \frac{1}{2}\tilde{L}(\|u\|_{L^2(\D)}+\delta)\psi(n)\left(2\|y\|_{\Gamma^{-1}}+2\|\F(u)\|_{\Gamma^{-1}}+\tilde{L}(\|u\|_{L^2(\D)}+\delta)\psi(n) \right).
\end{align}
From \eqref{eq:phi bound 1} we see that we can pick $\delta$ small enough such that $|\Phi(u)-\Phi(u')|< \frac{\eps}{2}$ for any $u'\in B_{\delta}(u)\subset \H$, and from \eqref{eq:phi bound 2} we can pick $N$ large enough such that, for any $n\geq N,$ we have $|\Phi(u')-\Phi^{(n)}(u')|< \frac{\eps}{2}$ for any $u'\in B_{\delta}(u)\subset \H$. Consequently,
there exist $N$ and $\delta$ such that, for any $n\geq N$ and $u'\in B_\delta(u)\in \H$, 
\begin{align}
    \left|\Phi(u)-\Phi^{(n)}(u')\right|<\eps,
\end{align}
proving continuous convergence of $\Phi^{(n)}$ to $\Phi$ as $n\to \infty$. We can then apply Theorem \ref{thm:Gamma Convergence} and complete the proof.
\end{proof}

\end{document}

%% file: ex_shared.tex
\usepackage{lipsum}
\usepackage{amsfonts}
\usepackage{graphicx}

\usepackage{enumitem}

\newcommand{\nc}{\normalcolor}

\newsiamremark{remark}{Remark}
\newsiamremark{hypothesis}{Hypothesis}
\crefname{hypothesis}{Hypothesis}{Hypotheses}
\newsiamthm{claim}{Claim}

\headers{Analysis of a Computational Framework for Bayesian Inverse Problems}{D. Sanz-Alonso and N. Waniorek}

\title{Analysis of a Computational Framework for Bayesian Inverse Problems: Ensemble Kalman Updates and MAP Estimators Under Mesh Refinement}

\author{Daniel Sanz-Alonso\thanks{Department of Statistics, University of Chicago, Chicago, IL 
  (\email{sanzalonso@uchicago.edu}}).
  \and Nathan Waniorek\thanks{Committee on Computational and Applied Mathematics, University of Chicago, Chicago, IL 
  (\email{waniorek@uchicago.edu}}).
 }

\newcommand{\R}{\mathbb{R}}

\newcommand{\N}{\mathbb{N}}

\usepackage[font = small, margin=30pt]{caption}

\newcommand{\Expect}{\operatorname{\mathbb{E}}}
\newcommand{\V}{\operatorname{\mathbb{V}}}

\newcommand{\M}{\mathcal{M}}

\renewcommand{\H}{\mathcal{H}}
\renewcommand{\V}{\mathcal{V}}
\renewcommand{\M}{\mathcal{M}}

\newcommand{\G}{\mathcal{G}}
\newcommand{\F}{\mathcal{F}}

\newcommand{\Cm}{\mathcal{C}}
\newcommand{\B}{\mathcal{B}}
\newcommand{\msK}{\mathscr{K}}

\newcommand{\D}{\mathcal{D}}
\newcommand{\Nc}{\mathcal{N}}

\newcommand{\iid}{\stackrel{\text{i.i.d.}}{\sim}}

\let\oldphi\phi
\renewcommand{\phi}{\varphi}
\newcommand{\eps}{\varepsilon}

  \newtheorem{example}[theorem]{Example}
    \newtheorem{assumption}[theorem]{Assumption}

    \usepackage{mathtools}

    \usepackage{amsmath,amssymb}

    \usepackage{mathrsfs} 

    \usepackage{bm} 

\usepackage{bbold}